\documentclass[11pt]{amsart}
\usepackage[english]{babel}
\usepackage{graphics}
\usepackage{amsfonts}
\usepackage{amsmath}
\usepackage{amssymb}
\usepackage{mathrsfs}
\usepackage{enumerate}
\usepackage{relsize,exscale}
\usepackage{cite}
\usepackage[colorlinks=true,urlcolor=blue,
citecolor=red,linkcolor=blue,linktocpage,pdfpagelabels,
bookmarksnumbered,bookmarksopen,pagebackref]{hyperref}
\usepackage[hyperpageref]{backref}
\usepackage{esint}
\usepackage{epsfig}
\usepackage{amscd}
\usepackage{amsmath}
\usepackage{graphicx}
\usepackage{newlfont}
\usepackage{latexsym}
\usepackage{amsthm}
\usepackage{color}

\newtheorem{theorem}{Theorem}[section]
\newtheorem{lemma}[theorem]{Lemma}

\theoremstyle{remark}
\newtheorem{remark}[theorem]{Remark}

\numberwithin{equation}{section}

\theoremstyle{definition}
\newtheorem{definition}[theorem]{Definition}

\newcommand{\Eucap}{\mathrm{Cap}}
\newcommand{\Lip}{\mathrm{Lip}}

\newcommand{\R}{\mathbb{R}}

\newcommand{\N}{\mathbb{N}}

\renewcommand{\d}{\mathrm{d}}

\newcommand{\RN}{\mathbb{R}^n}

\newcommand{\de}{\partial}

\makeindex
\begin{document}

\title[A symmetry result for cooperative elliptic systems with singularities]{A symmetry result for cooperative elliptic systems \\ with singularities}

\thanks{The authors are members of INdAM/GNAMPA.
The second author is supported by
the Australian Research Council Discovery Project 170104880 NEW ``Nonlocal
Equations at Work''. }

\author[S.\ Biagi]{Stefano Biagi}
\author[E.\ Valdinoci]{Enrico Valdinoci}
\author[E.\ Vecchi]{Eugenio Vecchi}

\address[S.\,Biagi]{Dipartimento di Ingegneria Industriale e Scienze Matematiche 
\newline\indent Universit\`a Politecnica della Marche \newline\indent
Via Brecce Bianche, 60131, Ancona, Italy}
\email{s.biagi@dipmat.univpm.it}

\address[E.\,Valdinoci]{Department of Mathematics and Statistics
\newline\indent University of Western Australia \newline\indent
35 Stirling Highway, WA 6009 Crawley, Australia}
\email{enrico.valdinoci@uwa.edu.au}

\address[E.\,Vecchi]{Dipartimento di Matematica \newline\indent
	Universit\`a degli Studi di Trento \newline\indent
	Via Sommarive 14, 38123, Povo (Trento), Italy}
\email{eugenio.vecchi@unitn.it}



\date{\today}

\begin{abstract}
We obtain symmetry results for solutions of an elliptic system
of equation possessing a cooperative structure.
The domain in which the problem is set may possess ``holes''
or ``small vacancies'' (measured in terms of capacity)
along which the solution may diverge.

The method of proof relies on the moving plane technique,
which needs to be suitably adapted here to take care of the complications
arising from the vacancies in the domain and the analytic
structure
of the elliptic system.
\end{abstract}

\maketitle

 \section{Introduction and main results} \label{sec.intromain}

The moving plane method was introduced in
the pioneer works of Aleksandrov \cite{Alek,Alek2} in order to characterize spheres as the only closed, smooth and connected surfaces having constant mean curvature.
Afterwards, starting from the seminal paper of Serrin \cite{Serrin} concerning the overdetermined torsion problem, 
Gidas, Ni and Nirenberg \cite{GNN} and Berestycki and Nirenberg \cite{BN} developed further this technique
in order to establish some qualitative properties of solutions of elliptic partial differential equations such as symmetry and mo\-no\-to\-ni\-ci\-ty.
The method of proof is very elegant, it relies on a beautiful geometric intuition, and its essential ingredient is the appropriate use of the maximum principle in comparing the values of the solution of the equation at two different points after a suitable reflection, which is determined by a hyperplane which gets moved up to a critical position. 

\medskip

In this paper, we exploit the moving plane technique in order to obtain symmetry results in a setting which is not usually comprised by the classical method, since two difficulties will be taken into account. First of all, we will consider the case of general cooperative elliptic systems rather than that of a single equation, for which the moving plane technique has been settled by Troy \cite{Troy81}.
This setting is also motivated by equations driven by polyharmonic operators with Navier boundary conditions (which, up to repeated substitutions, can be framed into elliptic systems of
second order equations). Moreover, we take into account the case in which the domain presents ``holes'', or ``cuts'', or more general vacancies, along which the solution can become singular. This is an extension of our previous work \cite{BVV}
where we were dealing only with singularities made out of a single point, as studied in \cite{Terracini96,CLN2} for the case of a single scalar equation.\medskip

Of course, one cannot expect a general treatment of these two situations without 
ad\-di\-tio\-nal assumptions. Indeed, general elliptic systems do not satisfy the maximum principle and there is no natural order in the vectorial case, making the classical regularity theories fail in such a situation. Moreover, if the vacancies in the domain are too large, they can affect the geometry involved in the reflections and produce singularities that cannot be treated analytically in any convenient way.\medskip

To overcome these difficulties, inspired by the recent works \cite{EFS,Sciunzi17}, we will restrict ourselves to the case of cooperative systems, in which an appropriate use of the maximum principle is possible, and consider domain vacancies that are ``sufficiently small'', in terms of capacities.\medskip

The precise mathematical formulation in which we work is the following. 
 Let $m\geq 2$ be a fixed natural number. Throughout the present paper,
 we shall be concerned with second-order cooperative (elliptic)
 systems of the following form
 \begin{equation} \label{eq.mainsystem}
 \begin{cases}
  -\Delta u_i = f_i(u_1,\ldots,u_m), & \text{in $\Omega\setminus\Gamma$}, \\
  u_i > 0, & \text{in $\Omega\setminus\Gamma$}, \\
  u_i \equiv 0, &\text{on $\de\Omega$},
 \end{cases}
 \end{equation}
 where $\Omega,\,\Gamma$ and $f_1,\ldots,f_m$ satisfy assumptions 
 (H.1)-to-(H.3) below:
 \begin{itemize}
  \item[(H.1)] $\Omega\subseteq\RN$ is a convex open set of class $C^\infty$ which 
  is bounded and symmetric with respect to the hyperplane $\mathit{\Pi} := \{x_1 = 0\}$; \medskip
  
  \item[(H.2)] $\Gamma\subseteq\Omega\cap\mathit{\Pi}$ is
  a closed set consisting of a point,
  if $n = 2$, or verifying
  \begin{equation} \label{eq.GammaCapzero}
  \text{$\Eucap_2(\Gamma) = 0$,\,\,if $n\geq 3$}.
  \end{equation}
  \item[(H.3)] $f_1,\ldots,f_m\in\Lip(\R^m)$ and, for every
  $i,j\in\{1,\ldots,m\}$ with $i\neq j$, the map
   $$\R \ni t_j\mapsto f_i(t_1,\ldots,t_{j-1},t_j,t_{j+1},\ldots,t_m)$$
   is non-decreasing on $(0,\infty)$ for every 
   choice of $t_1,\ldots,t_{j-1},t_{j+1},\ldots,t_n
   > 0$.
 \end{itemize}
We refer to Definition \ref{eq.defisolmain} for the rigorous definition of
solution used in this paper. See also Definition \ref{rem.ipotesiH2} 
for the precise meaning of {\it capacity of a set} and 
a detailed explanation of the assumption (H.2). We want to point out that
the capacitary assumption (H.2) cannot be removed nor replaced
with the request that $\mathcal{L}^{n}(\Gamma) = 0$, see Remark \ref{OnH2}.

  We are ready to state the main result of this paper.
 \begin{theorem} \label{thm.MAIN}
   Let $\Omega\subseteq\RN$ and $\Gamma\subseteq\Omega$ fulfill, 
   respectively,
   assumptions \emph{(H.1)} and \emph{(H.2)}.
   Moreover, let $f_1,\ldots,f_m$ satisfy assumption \emph{(H.3)}
   and let
   $$U = (u_1,\ldots,u_m)\in H^1_{loc}(\Omega\setminus\Gamma;\R^m)
   \cap C(\overline{\Omega}\setminus\Gamma;\R^m)$$
   be a (vec\-tor\--va\-lued) solution of the elliptic system
  \eqref{eq.mainsystem}.
  
  Then, $u_1,\ldots,u_m$ are symmetric with respect to the hyperplane
  $\mathit{\Pi}$ and increasing in the $x_1$-direction on $\Omega\cap\{x_1<0\}$.
  Furthermore, for every $i\in\{1,\ldots,m\}$ one has
  \begin{equation} \label{eq.monotoneui}
   \frac{\de u_i}{\de x_1}(x) > 0\qquad\text{for every $x\in\Omega\cap\{x_1<0\}$}.
  \end{equation}
 \end{theorem}

The proof of Theorem \ref{thm.MAIN} is pretty much inspired by \cite{Sciunzi17,EFS}. The main
idea in there relies in proving the symmetry (and monotonicity) of the solution through a clever
use of integral estimates. To be more precise, given the function $u$ and its reflection
across a given hyperplane, one considers the positive part of their difference and shows
that its gradient is actually 0. 
Passing to elliptic systems this technique becomes more involved because the presence
of more equations naturally leads to {\it interactions} between the solutions which
have to be carefully treated. Indeed,
these interaction between the different components of the (vectorial) solution, 
causes an important loss of information on the single equations. 
To overcome this difficulty we will implement a sort of bootstrap procedure in which an estimate on a single component 
is reflected into the next one, thus producing an iterative procedure 
that eventually leads to a closed formula valid for all the components of the solution. 
We also want to stress that our result extends our previous result in \cite{BVV} and it
is general enough to cover a bunch of
polyharmonic semilinear problems with Navier boundary conditions, even allowing for
possibly singular terms.

\medskip

The literature concerning symmetry results for
elliptic PDEs is pretty wide and this makes it hard for us 
to present here an exhaustive list of references.
We already mentioned the seminal papers 
\cite{Serrin, GNN, BN} for the introduction and the use of the moving planes method
in the elliptic PDEs setting. 
More recently, there has been an increasing interest in the study of
elliptic PDE's (in bounded domains $\Omega \subset \R^n$) allowing for possible singularities, namely PDE's of the form
\begin{equation*}
\left\{ \begin{array}{rl}
            -\Delta u = \tfrac{1}{u^{\gamma}} + g(u) & \textrm{in } \Omega,\\
						u=0 & \textrm{on } \partial \Omega,
					 \end{array}	\right.
\end{equation*}
\noindent with $\gamma >0$. In this perspective, we want to mention \cite{CRT},
which is one of the first contributions dealing with singular nonlinearities,
and then the more recent series of papers \cite{CD,CGS,CS,CMS}.\\

To the best of our knowledge, one of the first papers dealing with 
symmetry of positive solutions of elliptic PDE's in domains with {\it holes} given by
a single point, dates back to \cite{Terracini96}, which was then extended to slightly more general
operators and sets in \cite{CLN2}. The same kind of result, but with a necessary and delicate
modification of the technique involved, can be also obtained in presence of a {\it bigger hole}.
In this direction, we refer to \cite{Sciunzi17, EFS} where the authors allow (respectively)
for a hole given by a $n-2$-dimensional smooth manifold and a set of null capacity. 
Their ideas have also been successfully applied in the non-local setting, see \cite{MPS17}.\\
Let us now spend a few words concerning the case of (cooperative) elliptic systems, which can
also include the case of higher order polyharmonic PDE's with Navier boundary con\-di\-tions.
The first result aiming at extending the results in \cite{GNN} to the vectorial case
is contained in \cite{Troy81}. Subsequently, there has been an impressive amount of contributions
dealing with the validity of maximum principles (see e.g. \cite{deFigueiredo, Sirakov}).
Let us finally mention \cite{FGW, BGW08, DP, CoVe, CoVe2, BVV}
(for symmetry results for semilinear polyharmonic problems and cooperative
elliptic systems with or without singularities). 

\bigskip

The paper is organized as follows. In Section~\ref{sec.notaux} we fix the notation
used throughout the paper and we recall and prove a few technical results needed for the proof of
Theorem~\ref{thm.MAIN},
which is the content of the final Section~\ref{sec.proofThmMain}.

 \section{Notations and auxiliary results} \label{sec.notaux}
 The aim of this section is to introduce
 the relevant notations we shall need in the sequel,
 and to state some auxiliary results
 on which we shall base the proof of Theorem \ref{thm.MAIN}. 
   To begin with, we briefly review in this remark 
  the precise meaning of assumption (H.2) (in the meaningful
  case $n\geq 3$). 
	\begin{definition} \label{rem.ipotesiH2}
	If $U\subseteq\R^n$
  is open and $E\subseteq U$ is \emph{compact},
  the $2$-capacity of the condenser $(E,U)$ is defined as
  $$\mathrm{Cap}_2(E,U) := 
  \inf\left\{
  \int_U\|\nabla u\|^2\,dx:\,
  \text{$u\in C^\infty_0(U)$ and $u\geq 1$ on $E$}\right\}.$$
  We then say that $E$ has vanishing $2$-capacity
  (and we write $\Eucap_2(E) = 0$)
   if
  $$\mathrm{Cap}_2(E\cap U,U) = 0, \qquad\text{for every open set $U\subseteq\R^n$}.$$
	\end{definition}
	We recall that it can be easily proved that a compact set $E\subseteq\R^n$ has vanishing
  $2$-capacity if and only if there exists a \emph{bounded} 
  open neighborhood $U_0$ of $E$ such that
  $$\Eucap_2(E\cap U_0,U_0) = 0.$$
  For a demonstration of this fact we refer, e.g., to \cite[Lemma 2.9]{HKM}.\\

 We now specify what we mean by a solution of the system in~\eqref{eq.mainsystem}.
 \begin{definition} \label{eq.defisolmain}
 Under the above
 assumptions (H.1)-to-(H.3), we say that a vector-valued function 
 $U = (u_1,\ldots,u_m):\Omega\to\R^m$
 is a \textbf{solution} of the system in~\eqref{eq.mainsystem} if
 \begin{enumerate}
  \item $U\in H^1_{loc}(\Omega\setminus\Gamma;\R^m)\cap 
  C(\overline{\Omega}\setminus\Gamma;\R^m)$, that is, 
  $$u_i\in H^1_{loc}(\Omega\setminus\Gamma)\cap 
  C(\overline{\Omega}\setminus\Gamma),\qquad\text{for every
  $i = 1,\ldots,m$};$$
  \item  for every $i\in\{1,\ldots,m\}$ one has
  \begin{equation} \label{eq.defiuisolvesystem}
   \int_\Omega \langle\nabla u_i,\nabla\varphi\rangle\,\d x
   = \int_\Omega f_i(u_1,\ldots,u_m)\,\varphi\,\d x,\qquad\text{for every
   $\varphi\in C_0^\infty(\Omega\setminus\Gamma;\R)$};
  \end{equation}
  \item for every $i\in\{1,\ldots,m\}$ one has $u_i > 0$ a.e.\,on $\Omega$ and
  $u_i \equiv 0$ on $\de\Omega$.
 \end{enumerate}
 In this paper, if $U\subseteq\R^n$ is an arbitrary open set, the space
 $H_0^1(U)$ is intended as the closure of $C_0^\infty(U,\R)$ (or, equivalently, of
 $\Lip(U)\cap C_0(U,\R)$) with respect to the norm 
 $$\|u\|_{H^1(U)} := \|u\|_{L^2(U)}+\bigg(\int_U\|\nabla u\|^2\,\d x\bigg)^{1/2}.$$
 \end{definition}
 \begin{remark} \label{rem.regui}
 We point out that, on account of
 assumption (H.3), the right-hand side of any equation
 of the system in~\eqref{eq.mainsystem} is locally bounded;
 as a consequence, if $U = (u_1,\ldots,u_m)$ is a solution
 of this system of PDEs, from
 standard elliptic regularity we infer that
  $$u_1,\ldots,u_m\in C^{1,\alpha}_{loc}(\Omega\setminus\Gamma;\R),\qquad\text{for
  every $0<\alpha< 1$}.$$
  As a consequence, by condition (3) in Definition \ref{eq.defisolmain} 
  we have $u_i > 0$ \emph{for every} $x\in\Omega$.
 \end{remark}
 
 \noindent  We are now ready to set the standing notation
needed to perform the moving plane technique.
If $\Omega\subseteq\RN$ satisfies
 assumption (H.1), we set
 $$\mathbf{a}_\Omega := \inf_{\Omega}x_1.$$
 Moreover, for every fixed $\lambda\in\R$, we define
 $$\Sigma_\lambda := \{x\in\Omega:\,x_1<\lambda\},$$
 and we denote by $R_\lambda$ the symmetry with respect
 to the 
 hyperplane $\mathit{\Pi}_\lambda := \{x_1 = \lambda\}$, i.e.,
 $$R_\lambda:\RN\longrightarrow\RN, \qquad
 R_\lambda(x) = x_\lambda := (2\lambda-x_1,x_2,\ldots,x_n).$$
 We explicitly notice that, since $\Omega$ is open, then the same is true
 of $\Omega_\lambda := R_\lambda(\Omega)$; fur\-ther\-mo\-re, since
 $\Omega$ is convex, we clearly have that $\Sigma_\lambda$ is convex and 
 $\Sigma_\lambda\subseteq\Omega\cap\Omega_\lambda$.
 We collect in the next Lemma \ref{lem.topologico}
 some topological facts we shall need in the sequel.
 \begin{lemma} \label{lem.topologico}
  The following assertions hold true:
  \begin{itemize}
   \item[(1)] if $E\subseteq\R^n$ is a compact set with vanishing
  $2$-capacity and if $U\subseteq\R^n$ is a convex open set, then
   $U\setminus E$ is \emph{(}path-\emph{)}connected;
   \item[(2)] for every fixed $\lambda\in(\mathbf{a}_\Omega,0)$ one has
   $$\Eucap_2(\gamma_\lambda) = 0, \qquad\text{where
   $\gamma_\lambda := \Omega\cap\{x_1=\lambda\}$}.$$
  \end{itemize}
 \end{lemma}
 \begin{proof}
  (1)\,\,First of all we observe that, since the set $E$ has vanishing $2$-capacity,
  for every open neighborhood $\mathcal{O}$ of $E$ one has
  (see Remark \ref{rem.ipotesiH2})
  \begin{equation} \label{eq.assurdoallafine}
   \mathrm{Cap}_2(E,\mathcal{O}) = 0.
  \end{equation}
  Let then $x_0\neq y_0\in U$ be fixed, and let
  $\mathcal{O}_0\subseteq\R^n$ be an open neighborhood of
  $E$ such that $x_0,y_0\notin{\mathcal{O}_0}$.
  Moreover, let $\rho > 0$ be so small that
  \begin{equation} \label{eq.tousetocomplete}
   y_0\notin B(x_0,\rho), \quad B(x_0,\rho)\subseteq U \quad\text{and}\quad
  B(x_0,\rho)\cap \mathcal{O}_0 = \varnothing.
  \end{equation}
  We claim that there exist
  a point $x\in B(x_0,\rho)\subseteq U$ such that
  \begin{equation} \label{eq.CLAIMtopologico}
   \text{the segment $[x,y_0]$ joining $x$ to $y_0$ does not intersect
   $E$.}
  \end{equation}
  \noindent Taking this claim for granted for a moment, we are able to complete
  the demonstration of this assertion: in fact, if $x\in B(x_0,\rho)$
  is as in \eqref{eq.CLAIMtopologico}, the polygonal
  $$c := [x_0,x]\cup[x,y_0]$$
  connects $x_0$ to $y_0$ and it is contained in $U\setminus E$
  (this is a straightforward consequence
  of \eqref{eq.tousetocomplete}, \eqref{eq.CLAIMtopologico} and
  of the fact that, by assumption, $U$ is convex).
  
  We now turn to prove the above claim. To this end, we argue by contradiction
  and we assume that, for every fixed $x\in B(x_0,\rho)$, there exists
  $\overline{t} = \overline{t}_x\in(0,1)$ such that
  \begin{equation} \label{eq.assurdo}
   x+\overline{t}_x(y_0-x)\in E.
   \end{equation}
  If $u\in C_0^\infty(\mathcal{O}_0,\R)$ is any smooth function
  satisfying $u\geq 1$ on $E$, by combining
  \eqref{eq.tousetocomplete} with \eqref{eq.assurdo} we obtain the following
  estimate
  (note that $x\notin\mathcal{O}_0\supset\mathrm{supp}(u)$):
  \begin{equation} \label{eq.tointegrate}
   \begin{split}
    1 & \leq u\big(x+\overline{t}_x(y_0-x)\big) =
   u\big(x+\overline{t}_x(y_0-x)\big) - u(x) \\[0.2cm]
   & \quad = \int_0^{\overline{t}_x}\langle (\nabla u)\big(x+s(y_0-x)\big),y_0-x\rangle\,\d s
   \\[0.2cm]
   & \quad \leq \|y_0-x\|\,\int_0^{\overline{t}_x}\|(\nabla u)\big(x+s(y_0-x)\big)\|\,\d s \\[0.2cm]
   & \quad \big(\text{by H\"older's inequality, and setting
   $\kappa_0 := \|y_0-x_0\|+\rho$}\big) \\[0.2cm]
   & \quad \leq \kappa_0\,
   \bigg(\int_0^{\overline{t}_x}\|(\nabla u)\big(x+s(y_0-x)\big)\|^2\,\d s\bigg)^{1/2}.
   \end{split}
  \end{equation}   
  Due to the arbitrariness of $x\in B(x_0,\rho)$, we are entitled
  to integrate both sides
  of \eqref{eq.tointegrate} on $B(x_0,\rho)$ with respect to $x$:
  this gives (with $\omega_n := |B(0,1)|$)
  \begin{align*}
   \omega_n\,\rho^n & \leq 
   \kappa_0^2\,\int_{B(x_0,\rho)}
   \bigg(\int_0^{\overline{t}_x}\|(\nabla u)\big(x+s(y_0-x)\big)\|^2\,\d s\bigg)\,\d x \\[0.2cm]
   & \leq
   \kappa_0^2\,\int_{\R^n}
   \bigg(\int_0^{1}\|(\nabla u)\big(x+s(y_0-x)\big)\|^2\,\d s\bigg)\,\d x \\[0.2cm]
   & = \kappa_0^2\,
   \int_{\R^n}\|\nabla u\|^2\,\d x.
  \end{align*}
  Since the function $u$ was arbitrary, the above estimate implies that
  $$\inf\bigg\{\int_{\R^n}\|\nabla u\|^2\,\d x:\,\text{$u\in C^\infty_0(\mathcal{O}_0,\R)$
  and $u\geq 1$ on $E$}\bigg\}\geq 
  \frac{\omega_n\,\rho^n}{\kappa_0^2},$$
  but this is in contradiction with \eqref{eq.assurdoallafine}. Thus,
  \eqref{eq.CLAIMtopologico} holds. \medskip
  
  (2)\,\,If $n = 2$, from the convexity
  of $\Omega$ (and the fact that, by assumption, $\lambda > 
  \mathbf{a}_\Omega$) it readily follows that
  $\gamma_\lambda$ consists exactly of two points; as a consequence,
  $$\mathrm{Cap}_2(\gamma_\lambda) = 0.$$
  If, instead, $n\geq 3$, we claim that 
  \begin{equation} \label{eq.gammalambdamanif}
   \text{$\gamma_\lambda$ is a smooth
  $(n-2)$-dimensional manifold.}
  \end{equation} 
  Taking this claim for granted for a moment,
  we are able to complete the proof of the statement: indeed, 
  on account of \eqref{eq.gammalambdamanif},
  it is readily seen that the Hausdorff dimension of $\gamma_\lambda$ is precisely
  $n-2$; as a consequence, we have (see, e.g., \cite{EG})
  $$\mathrm{Cap}_2(\gamma_\lambda) = 0.$$
  We then turn to prove \eqref{eq.gammalambdamanif}. To this end, let
  $\xi\in\gamma_\lambda$ be fixed. Since $\Omega$ is an
  open set of class $C^\infty$
  (see assumption (H.1)), there exist an index $i\in\{1,\ldots,n\}$,
  a number $\rho > 0$ and a map  
  $\theta \in C^\infty(B(\xi',\rho),\R)$ 
  (where $\xi' = (\xi_1,\ldots,\xi_{i-1},\xi_{i+1},\ldots,\xi_n)$) such that
  \begin{equation*}
   \begin{split}  
   & \de\Omega\cap\big((\xi_i-\rho,\xi_i+\rho)\times B(\xi',\rho)\big) = \\[0.2cm]
  & \quad = \big\{x = (x_i,x')\in (\xi_i-\rho,\xi_i+\rho)\times B(\xi',\rho):\,x_i = \theta(x')\big\}.
  \end{split}
  \end{equation*}
  Moreover, since $\Omega$ is convex and $\lambda > \mathbf{a}_\Omega$, it is
  quite easy to recognize that $\theta$ is either convex or concave
  on $B(\xi',\rho)$ and that, setting 
  $g(x) = g(x_i,x') := x_i-\theta(x')$,
  $$\text{$\nabla g(\xi)$ is not parallel to $e_1 = (1,0,\ldots,0)$}.$$
  As a consequence, if we introduce the $\R^2$-valued function
  $$\alpha(x) = \alpha(x_i,x') := \big(x_i-\theta(x'),x_1-\lambda\big)
  = \big(g(x),x_1-\lambda\big)$$
  (with $ x = (x_i,x') = (\xi_i-\rho,\xi_i+\rho)\times B(\xi',\rho)$), we clearly have
  that \medskip
  
  (a)\,\,$\alpha$ is smooth on $\mathcal{U} := (\xi_i-\rho,\xi_i+\rho)\times B(\xi',\rho)$; \medskip
  
  (b)\,\,the Jacobian matrix of $\alpha$ at $\xi$ has full rank; \medskip
  
  (c)\,\,$\gamma_\lambda\cap\mathcal{U} := \{x\in\mathcal{U}:\,\alpha(x) = 0\}$. \medskip
  
  \noindent Gathering together all these facts, we conclude that
  $\gamma_\lambda$ is a smooth manifold
  of dimension $n-2$, and the proof
  is finally complete.
  \end{proof}
  \begin{remark} \label{rem.casofacile}
   We explicitly observe that, on account of
   Lemma \ref{lem.topologico}-(1), we have that
   \begin{equation} \label{eq.SigmamenoRGammaconn}
    \text{$\Sigma_\lambda\setminus R_\lambda(\Gamma)$ is connected
   for every $\lambda\in(\mathbf{a}_\Omega,0)$}.
   \end{equation}
   In fact, since $\Gamma$ fulfills (H.2), we have
   that $R_\lambda(\Gamma)$ is compact and
   $\Eucap_2(R_\lambda(\Gamma)) = 0$ (for e\-ve\-ry $n\geq 2$); moreover, as $\Omega$ is convex,
   the same is true of $\Sigma_\lambda = \Omega\cap\{x_1<\lambda\}$.
   
   Actually, \eqref{eq.SigmamenoRGammaconn} can be proved in a more
   direct (and simpler) way by observing that
   \begin{equation} \label{eq.RGammacontainedplane}
    R_\lambda(\Gamma)\subseteq\{x_1 = 2\lambda\}.
    \end{equation}
   In fact, since $R_\lambda(\Gamma)$ has vanishing $2$-capacity,
   it is well-known that
   $$\mathcal{H}_{\text{dim}}(R_\lambda(\Gamma)) \leq n-2,$$
   where $\mathcal{H}_{\text{dim}}(R_\lambda(\Gamma))$ stands for
   the Hausdorff dimension of $R_\lambda(\Gamma)$ in $\R^n$ (see, e.g., \cite{HKM});
   as a consequence, 
   there necessarily exists (at least) one point 
   \begin{equation} \label{eq.spaziosulpiano}
    \text{$\overline{x}\in\Sigma_\lambda\cap\{x_1 = 2\lambda\}$
   such that $x\notin R_\lambda(\Gamma)$}.
   \end{equation}
   By combining \eqref{eq.RGammacontainedplane} with 
   \eqref{eq.spaziosulpiano} it is very easy to recognize that,
   if $x_0\neq y_0\in\Sigma_\lambda\setminus R_\lambda(\Gamma)$ are arbitrary,
   the polygonal
   $c = [x_0,\overline{x}]\cup [\overline{x},y_0]$ connects
   $x_0$ to $y_0$ and it lays in $\Sigma_\lambda\setminus R_\lambda(\Gamma)$.
  \end{remark}
 Let now $\Gamma\subseteq\Omega$ satisfy
 assumption (H.2), and let 
 $f_1,\ldots,f_m$ be as in assumption (H.3). 
 If $U = (u_1,\ldots,u_m):\Omega\to\R^m$ is any solution
 of the elliptic system \eqref{eq.mainsystem}
 (according to Definition \ref{eq.defisolmain}), we then introduce
 the following functions (defined on $\Omega_\lambda\setminus R_\lambda(\Gamma)$):
 \begin{equation} \label{eq.defiUlambda}
  u_i^{(\lambda)} := u_i\circ R_\lambda\qquad\text{and}\qquad
 U_\lambda := (u_1^{(\lambda)},\ldots,u_m^{(\lambda)}) = U\circ R_\lambda.
 \end{equation}
 On account of Remark \ref{rem.regui}, we clearly have (for every
 $0<\alpha< 1$)
 \begin{equation} \label{eq.regulUlambda}
   U_\lambda\in C^{1,\alpha}(\Omega_\lambda\setminus R_\lambda(\Gamma);\R^m)
   \cap C(\overline{\Omega_\lambda}\setminus R_\lambda(\Gamma);\R^m).
   \end{equation}
  Furthermore, since $U$ solves \eqref{eq.mainsystem}, we have
   \begin{equation} \label{eq.solvedbyUlambda}
    \begin{cases}
     -\Delta u_i^{(\lambda)}
     = f_i(u_1^{(\lambda)},\ldots,u_n^{(\lambda)}), & \text{in 
     $\Omega_\lambda\setminus R_\lambda(\Gamma)$}; \\[0.15cm]
     u_i^{(\lambda)} > 0, & \text{in $\Omega_\lambda\setminus R_\lambda(\Gamma)$}; \\[0.15cm]
     u_i^{(\lambda)} \equiv 0, & \text{on $\de\Omega_\lambda$}.
    \end{cases} 
   \end{equation}      
  We explicitly notice that, since $U_\lambda$ is not of
  class $C^2$, by saying that 
  $u_1^{(\lambda)},\ldots,u_n^{(\lambda)}$  
  solve the (system of) PDEs
  in \eqref{eq.solvedbyUlambda} we mean, precisely, that
  \begin{equation} \label{eq.solveduilambda}
    \int_{\Omega_\lambda} 
    \langle\nabla u_i^{(\lambda)},\nabla\varphi\rangle\,\d x
   = \int_{\Omega_\lambda} f_i(u_1^{(\lambda)},\ldots,
   u_m^{(\lambda)})\,\varphi\,\d x,
   \quad\text{$\forall\,\,\varphi\in C_0^\infty(\Omega_\lambda\setminus R_\lambda(\Gamma);\R)$}.
   \end{equation}
	
	\begin{remark}\label{OnH2}
	As already mentioned in the Introduction, assumption (H.2) is somehow sharp.
	Let us clarify this fact with a couple of examples in the scalar case. \medskip
	
  \noindent {\bf Example 1:} In Euclidean space $\R^n$, let $\Omega := B(0,1)$ and let
  $\Gamma := \overline{B(0,1/2)}$. Since all the boundary points
  of the annulus $\mathcal{O} := \Omega\setminus\Gamma$
  are regular for the Dirichlet problem for $\Delta$,
  there exists a unique function $u\in C^\infty(\mathcal{O},\R)\cap C(\overline{\mathcal{O}},\R)$
  such that
  $$\begin{cases}
   \Delta u = 0 & \text{in $\mathcal{O} = \Omega\setminus\Gamma$}, \\
   u \equiv 0 & \text{on $\de\Omega$}, \\
   u(x) = e^{x_1} & \text{for every $x\in\de B(0,1/2)$}.
  \end{cases}
  $$
  Owing to the classical weak and strong maximum principles, it is readily
  seen that $u > 0$ on $\mathcal{O} = \Omega\setminus\Gamma$; moreover, since $u$ is continuous up
  to $\overline{\mathcal{O}}$ and since $x\mapsto e^{x_1}$ is not 
  even in $x_1$, we infer that $u$ cannot be symmetric with respect
  to the hyperplane $\mathit{\Pi} = \{x_1 = 0\}$. \vspace*{0.12cm}
  
  Summing up, the function $u$ is a solution
  of \eqref{eq.mainsystem} (with $m = 1$ and $f\equiv 0$) which \emph{is not}
  symmetric with respect to the hyperplane  $\{x_1 = 0\}$. Notice
  that both $\Omega$ and $\Gamma$ are symmetric w.r.t.\,$\{x_1 = 0\}$, 
  but $\Gamma$ has not vanishing $2$-capacity
  (since $|\Gamma| > 0$). \medskip
	
  \noindent {\bf Example 2:} 
  In Euclidean space $\R^2$, let $\Omega := B(0,1)$
  and let 
  $\Gamma := \{0\}\times [-1/2,1/2].$
  Moreover, for every fixed $n\geq 2$, we consider the 
  (closed) rectangle
  $$R_n := [-1/n,1/n]\times[-1/2,1/2],$$ 
  and we choose a function $\varphi_n\in \Lip(R_n)$ such that 
  $$\text{$\varphi_n\equiv 1$ on $\{1/n\}\times[-1/2,1/2]$\quad and \quad
  $\varphi_n\equiv 2$ on $\{-1/n\}\times[-1/2,1/2]$.}$$
  Finally, we define $\Omega_n := \Omega\setminus R_n$.
  Since $\Omega_n$ is regular for the Dirichlet problem
  for $\Delta$, it is possible to find
  a unique function $u_n\in C^\infty(\Omega_n,\R)
  \cap C(\overline{\Omega_n},\R)$ such that
  $$\begin{cases}
   \Delta u_n = 0 & \text{in $\Omega_n$}, \\
   u_n\equiv 0 & \text{on $\de\Omega$}, \\
   u_n \equiv \varphi_n & \text{on $\de R_\epsilon$}.
  \end{cases}
  $$
  Furthermore, by the classical weak and strong maximum principles we have 
  \begin{equation} \label{eq.ununifbounded}
   \text{$0\leq u_n\leq 2$ on $\overline{\Omega_n}$ \quad and 
  \quad $u_n > 0$ on $\Omega_n$}.
  \end{equation}
  We claim that the sequence $\{u_n\}_{n}$ has a cluster point
  $u_0$ which is a solution of \eqref{eq.mainsystem}
  (with $m = 1$ and $f\equiv 0$) but which is not symmetric with respect
  to the hyperplane $\{x_1 = 0\}$. \medskip
  
  To prove the claim we first observe that, if $k\in\N$ is arbitrarily fixed and if
  $$\mathcal{O}_k := \big\{
  x\in\Omega\setminus\Gamma:\,d\big(x,\de(\Omega\setminus\Gamma)\big) > 1/k\big\},$$
  there exists a natural $n_k\geq 2$ such that $\overline{\mathcal{O}_k}
  \subseteq\Omega_{n}$ for every $n\geq n_k$.
  As a consequence, since $\{u_n\}_{n\geq n_k}$ is a sequence of harmonic functions
  in $\mathcal{O}_k$ which is uniformly bounded on $\mathcal{O}_k$, there exists
  a harmonic function $u_{0k}$ on $\mathcal{O}_k$ such that
  (up to a sub-sequence)
  $$\lim_{n\to\infty}u_n = u_{0k}, \quad \text{uniformly on every
  compact set of $\mathcal{O}_k$}.$$
  From this, by exploiting a suitable Cantor diagonal argument, it is then possible to find
  a sub-sequence $\{u_{n_j}\}_{j}$ of $\{u_n\}_n$ and a harmonic function
  $u_0$ on $\Omega\setminus\Gamma$ such that
  $$\lim_{j\to\infty}u_{n_j} = u_{0}, \quad \text{uniformly on every
  compact set of $\Omega\setminus\Gamma$}.$$
  In particular, since $u_n\equiv 0$ on $\de\Omega$ 
  and $u_n > 0$ on $\Omega_n$ for every $n\in\N$, we infer that
  $$\text{$u_0\equiv 0$ on $\de\Omega$\quad and \quad
  $u_0\geq 0$ on $\Omega\setminus\Gamma$.}$$
  Let now $n\geq 2$ be arbitrarily fixed, let $P_n := (-1/n,0)$ and let 
  $$B_n^- := B(P_n,1/4)\cap\{x_1<-1/n\}\subseteq\Omega_n.$$ 
  Since $\varphi_n$ is Lipschitz-continuous on $R_n$ and since
  $B_n^-\cap\{x_1 = -1/n\}$ is a Lipschitz portion of $\de B_n^-$, it follows from
  classical results (see, e.g., Theorem~4.11
   in~\cite{GT}) that
  $$|u_n(x)-2| = |u_n(x) - u_n(P_n)| \leq C\,|x_1+1/n|,
  \quad \text{for any $x = (x_1,0)\in B_n^-$},$$
  where $C$ is a suitable positive constant which is \emph{independent of} $n$.
  From this, by letting $n\to\infty$ (and reminding that $u_{n_j}\to u_0$ as $j\to\infty$
  point-wise on $\Omega\setminus\Gamma$) we get
  $$|u_0(x)-2|\leq C|x_1|, \quad \text{for every $x = (x_1,0)\in\Omega\setminus\Gamma$
  with $x_1 < 0$}.$$
  As a consequence, we infer that
  \begin{equation} \label{eq.limitesinistro}
   \exists\,\,\lim_{\begin{subarray}{c}
   x\to 0 \\
   x_1 < 0
  \end{subarray}}u_0(x) = 2.
  \end{equation}
  On the other hand, if $Q_n := (1/n,0)$ and if
  $$B_n^+ := B(Q_n,1/4)\cap\{x_1>1/n\}\subseteq\Omega_n,$$
  by arguing exactly as before we get 
  $$|u_n(x)-1| = |u_n(x) - u_n(Q_n)| \leq C'\,|x_1-1/n|,
  \quad \text{for any $x = (x_1,0)\in B_n^+$},$$
  where $C'$ is another positive constant which is \emph{independent of} $n$.
  From this, by letting $n\to\infty$ and by taking the limit
  as $x\to 0$ with $x_1 > 0$, we obtain
  \begin{equation} \label{eq.limitedestro}
   \exists\,\,\lim_{\begin{subarray}{c}
   x\to 0 \\
   x_1 > 0
  \end{subarray}}u_0(x) = 1.
  \end{equation}
  Gathering together \eqref{eq.limitesinistro} and
  \eqref{eq.limitedestro} we readily see that
  $u_0$ \emph{cannot be} symmetric with respect to the hyperplane
  $\{x_1 = 0\}$; moreover, since $u_0$ is harmonic and non-negative
  on $\Omega\setminus\Gamma$, by the strong maximum principle we conclude that
  $u_0 > 0$ on $\Omega\setminus\Gamma$.\vspace*{0.12cm}
  
  Summing up, $u_0$ is a solution of \eqref{eq.mainsystem}
  (with $m = 1$ and $f\equiv 0$) which is not symmetric
  with respect to the hyperplane $\{x_1 = 0\}$. Note that, even if
  $|\Gamma| = 0$, the set $\Gamma$ cannot have vanishing $2$-capacity: in fact,
  its Hausdorff dimension is strictly greater than
  $n - 2 = 0$.
	\end{remark}
	
   \medskip
   
 After these preliminaries, we continue this section by constructing
 two sequences of functions which shall play a fundamental r\^ole
 in the proof of Theorem \ref{thm.MAIN}. In order to do this, we exploit
 some ideas contained in \cite{EFS} (see, precisely, Section 2).
 
 First of all we observe that, if $\lambda\in(\mathbf{a}_\Omega,0)$ is arbitrarily fixed,
 on account of Lemma \ref{lem.topologico}-(2) we have
 $\Eucap_2(R_\lambda(\Gamma)) = 0$ (both in the case
 $n = 2$ and in the case $n \geq 3$); as a consequence, 
 if $\mathcal{O}\subseteq\RN$ is any open neighborhood of $R_\lambda(\Gamma)$,
 we have 
 \begin{equation} \label{eq.RlGammaopen}
  \Eucap_2\big(R_\lambda(\Gamma),\mathcal{O}\big) = 0.
 \end{equation}
 On account of \eqref{eq.RlGammaopen}, for any
 $k\in\N$ it is possible to find a
 function $\widetilde{\psi}_k\in C_0^\infty(\R^n,\R)$ (also depending
 on the fixed $\lambda$) such that
 \begin{itemize}
 \item $\widetilde{\psi}_k\geq 1$ on $R_\lambda(\Gamma)$ and
  $\mathrm{supp}(\widetilde{\psi}_k)\subseteq\mathcal{O}_k$, where
  $$\mathcal{O}_k = \Big\{x\in\R^n:\,d\big(x,R_\lambda(\Gamma)\big)<2^{-k}\Big\};$$
  \item $\int_{\RN}|\nabla\widetilde{\psi}_k|^2\,\d x \leq {1}/{k}$.
 \end{itemize}
 \medskip
 
 \noindent Starting from the sequence $\{\widetilde{\psi}_k\}_{k\in\N}$, we then define
 \begin{equation} \label{eq.defipsikT}
 \psi_k := T\circ\widetilde{\psi}_k, \qquad\text{where}\,\,
 T(s) := \begin{cases}
 1, & \text{if $s < 0$}, \\
 1-2s, & \text{if $0\leq s\leq 1/2$}, \\
 0, & \text{if $s > 1/2$}.
 \end{cases}
 \end{equation}
 Clearly, $\{\psi_k\}_{k\in\N}\subseteq\Lip(\R^n)$ and, for every fixed
 $k\in\N$, one has
 \begin{equation}\label{eq.propertiespsik}
 0\leq\psi_k\leq 1, \quad
 \text{$\psi_k\equiv 1$ on $\R^n\setminus\mathcal{O}_k$}, \quad
 \text{$\psi_k\equiv 0$ on 
 a small neighborhood of $R_\lambda(\Gamma)$}.
 \end{equation}
 Furthermore, since $\nabla\psi_k = (T'\circ\widetilde{\psi}_k)\cdot
 \nabla\widetilde{\psi}_k$ a.e.\,on $\R^n$, we also have
 \begin{equation} \label{eq.properypsikintegral}
 \int_{\R^n}|\nabla\psi_k|^2\,\d x\leq \frac{4}{k}, \qquad\text{for every $k\in\N$}.
 \end{equation}
 Arguing analogously, we construct a second sequence of functions 
 $\{\phi_h\}_{h\in\N}$ such that, for every $h\in\N$, the function
 $\phi_h$ is identically $0$ near the set
 $$\gamma_\lambda = \de\Omega\cap\mathit{\Pi}_\lambda = \{x\in\de\Omega:\,x_1 = \lambda\}.$$
 To this we first remind that,
 by Lemma \ref{lem.topologico}-(2), we have
 $\Eucap_2(\gamma_\lambda) = 0$; as a consequence,
 for every open neighborhood $\mathcal{V}\subseteq\RN$ of $\gamma_\lambda$ one has
 $$\Eucap_2(\gamma_\lambda,\mathcal{V}) = 0.$$
 On account of this last fact, in correspondence
 to every natural $h$ it is possible
 to construct a function $\widetilde{\phi}_h\in C_0^\infty(\R^n,\R)$
 (also depending on the fixed $\lambda$) such that
  \begin{itemize}
 \item $\widetilde{\phi}_h\geq 1$ on $\gamma_\lambda$ and
  $\mathrm{supp}(\widetilde{\phi}_h)\subseteq\mathcal{V}_h$, where
  $$\mathcal{V}_h = \Big\{x\in\R^n:\,d\big(x,\gamma_\lambda\big)<2^{-h}\Big\};$$
  \item $\int_{\RN}|\nabla\widetilde{\phi}_h|^2\,\d x \leq {1}/{h}$.
 \end{itemize}
 Starting from the sequence $\{\widetilde{\phi}_h\}_{h\in\N}$, we define (as above)
 \begin{equation} \label{eq.defiphihT}
 \phi_h := T\circ\widetilde{\phi}_h, \qquad\text{where $T$ is as in \eqref{eq.defipsikT}.}
 \end{equation}
 Clearly, $\{\phi_h\}_{h\in\N}\subseteq\Lip(\R^n)$ and, for every fixed
 $h\in\N$, one has 
 \begin{equation}\label{eq.propertiesphih}
 0\leq\phi_h\leq 1, \quad
 \text{$\phi_h\equiv 1$ on $\R^n\setminus\mathcal{V}_h$}, \quad
 \text{$\phi_h\equiv 0$ on a small neighborhood of $\gamma_\lambda$},\qquad
 \end{equation}
 Furthermore, since $\nabla\phi_h = (T'\circ\widetilde{\phi}_h)\cdot
 \nabla\widetilde{\phi}_h$ a.e.\,on $\R^n$, we also have
 \begin{equation} \label{eq.properyphihintegral}
 \int_{\R^n}|\nabla\phi_h|^2\,\d x\leq \frac{4}{h}, \qquad\text{for every $h\in\N$}.
 \end{equation}
Having defined the sequences $\{\phi_h\}_{h\in\N}$
 and $\{\psi_k\}_{k\in\N}$, we conclude this section
 by stating some auxiliary results which shall be used
 to prove Theorem \ref{thm.MAIN}.
 To begin with, we state the following Lemmas \ref{lem.cutoffgammalambda} and
 \ref{lem.cutoffRlambda}, which are resemblant of \cite[Lemma 3.1]{EFS}.
 \begin{lemma} \label{lem.cutoffgammalambda}
 Let $\lambda\in (\mathbf{a}_\Omega,0)$ be such that
 $R_\lambda(\Gamma)\cap\overline{\Omega} = \varnothing$, and 
 let $\{\phi_h\}_{h\in\N}$ be the sequence
 defined in \eqref{eq.defiphihT}. 
 Moreover, let $g\in C^1(\Sigma_\lambda,\R)\cap C(\overline{\Sigma_\lambda},\R)$
 be such that
 $$\text{$g\equiv 0$ on $\de\Sigma_\lambda\cap\mathit{\Pi}_\lambda$}
 \qquad\text{and}\qquad
 \text{$g < 0$ on $\de\Sigma_\lambda\setminus\mathit{\Pi}_\lambda$}.$$
 Then, the sequence of functions $\{\varphi_h\}_{h\in\N}$ defined by
 $$\varphi_h(x) := \begin{cases}
 g^+(x)\,\phi^2_h(x), &\text{if $x\in\Sigma_\lambda$}, \\
 0, & \text{if $x\in\R^n\setminus\Sigma_\lambda$},
 \end{cases}$$
 \emph{(}here, $g^+ = \max\{g,0\}$ is the positive part
 of $g$\emph{)} satisfies the following properties:
 \begin{itemize}
  \item[{(i)}] $\{\varphi_h\}_{h\in\N}\subseteq \Lip(\RN)$;
  \item[{(ii)}] $\mathrm{supp}(\varphi_h)\subseteq
  (\Omega\setminus\Gamma)\cap(\Omega_\lambda\setminus R_\lambda(\Gamma))\cap
  \overline{\Sigma_\lambda}$ (for every fixed $h\in\N$);
  \item[{(iii)}] for every $h\in\N$, and a.e.\,on $\Omega\cup\Omega_\lambda$, one has
  \begin{equation} \label{eq.nablavarphihone}
   \nabla\varphi_h = \big[\phi^2_h\,(\mathbf{1}_{\{g > 0\}}\cdot
   \nabla g) + 2\,\phi_h\,g^+\cdot\nabla \phi_h\big]
   \mathbf{1}_{\mathrm{supp}(\varphi_h)}.
\end{equation}   
 \end{itemize}
 In particular, $\varphi_h\in
 \Lip(\overline{\Sigma_\lambda})$ and $\varphi_h\equiv 0$ on $\de\Sigma_\lambda$,
 so that $\varphi_h\in H_0^1(\Sigma_\lambda)$.
 \end{lemma}
  \begin{lemma} \label{lem.cutoffRlambda}
 Let $\lambda\in (\mathbf{a}_\Omega,0)$ be such that
 $R_\lambda(\Gamma)\cap\overline{\Omega} \neq \varnothing$, and 
 let $\{\psi_k\}_{k\in\N},\,\{\phi_h\}_{h\in\N}$ be the sequences
 defined, respectively, in \eqref{eq.defipsikT} and in \eqref{eq.defiphihT}.
 Moreover, let 
 $$g\in C^1(\Sigma_\lambda\setminus R_\lambda(\Gamma),\R)
 \cap C(\overline{\Sigma_\lambda}\setminus R_\lambda(\Gamma),\R),$$
 be such that
 $$\text{$g\equiv 0$ on $(\de\Sigma_\lambda\cap\mathit{\Pi}_\lambda)\setminus R_\lambda(\Gamma)$}
 \qquad\text{and}\qquad
 \text{$g < 0$ on $(\de\Sigma_\lambda\setminus\mathit{\Pi}_\lambda)\setminus R_\lambda(\Gamma)$}.$$
 Then, the (double) sequence of functions $\{\varphi_{h,k}\}_{h,k\in\N}$ defined by
 $$\varphi_{h,k}(x) := \begin{cases}
 g^+(x)\,\phi^2_h(x)\,\psi^2_k(x), &\text{if $x\in\Sigma_\lambda$}, \\
 0, & \text{if $x\in\R^n\setminus\Sigma_\lambda$},
 \end{cases}$$
 satisfies the following properties:
 \begin{itemize}
  \item[{(i)}] $\{\varphi_{h,k}\}_{h\in\N}\subseteq \Lip(\RN)$;
  \item[{(ii)}] $\mathrm{supp}(\varphi_{h,k})\subseteq
  (\Omega\setminus\Gamma)\cap(\Omega_\lambda\setminus R_\lambda(\Gamma))\cap
  \overline{\Sigma_\lambda}$ (for every fixed $h\in\N$);
  \item[{(iii)}] for every $h\in\N$, and a.e.\,on $\Omega\cup\Omega_\lambda$, one has
  \begin{equation} \label{eq.nablavarphihtwo}
   \nabla\varphi_{h,k} = 
   \big[\phi^2_h\,\psi^2_k\,(\mathbf{1}_{\{g > 0\}}\cdot
   \nabla g) + 2\,g^+(\psi^2_k\,\phi_h\cdot\nabla \phi_h
   +\phi^2_h\,\psi_k\cdot\nabla\psi_k)\big]
   \mathbf{1}_{\mathrm{supp}(\varphi_{h,k})}.
\end{equation}   
 \end{itemize}
 In particular, $\varphi_{h,k}\in
 \Lip(\overline{\Sigma_\lambda})$ and $\varphi_{h,k}\equiv 0$ on $\de\Sigma_\lambda$,
 so that $\varphi_{h,k}\in H_0^1(\Sigma_\lambda)$.
 \end{lemma}
 We also have the following regularity result for the solutions
 of \eqref{eq.mainsystem}, which can be demonstrated
 by arguing essentially as in the proof of \cite[Lemma 3.2]{EFS}.
 \begin{lemma} \label{lem.regularitywlambda}
  Let $\lambda\in(\mathbf{a}_\Omega,0)$ and $i\in\{1,\ldots,m\}$ be fixed.
  Then,
  \begin{equation}
   (u_i-u_i^{(\lambda)})^+\in H_0^1(\Sigma_\lambda).
  \end{equation}
  Furthermore, if $\lambda$ is such that $R_\lambda(\Gamma)\cap\overline{\Omega}
  = \varnothing$, then
  \begin{equation}
   \text{$\displaystyle\lim_{h\to\infty}
   \varphi_h = (u_i-u_i^{(\lambda)})^+$ in $H_0^1(\Sigma_\lambda)$},
   \qquad\begin{array}{c}
   \text{where $\varphi_h$ is as in Lemma \ref{lem.cutoffgammalambda},} \\
   \text{with $g = u_i-u_i^{(\lambda)}$}.
   \end{array} 
  \end{equation}
  If, instead, $\lambda$ is such that $R_\lambda(\Gamma)\cap\overline{\Omega}
  \neq \varnothing$, then
  \begin{equation}
   \text{$\displaystyle\lim_{
   \begin{subarray}{c}
   h,\,k\to\infty 
   \end{subarray}}
   \varphi_{h,k} = (u_i-u_i^{(\lambda)})^+$ in $H_0^1(\Sigma_\lambda)$},
   \qquad\begin{array}{c}
   \text{where $\varphi_{h,k}$ is as in Lemma \ref{lem.cutoffRlambda},} \\
   \text{with $g = u_i-u_i^{(\lambda)}$}.
   \end{array} 
  \end{equation}
 \end{lemma}
 
Finally, we prove a technical lemma which will be used in the proof
of Theorem \ref{thm.MAIN}.

\begin{lemma}\label{LemmaPoincare}
Let $n \geq 2$ and let $U \subseteq \R^n$ be an open and bounded set with Lipschitz boundary.
There exists a real constant $\Theta=\Theta_n>0$, independent
of $U$, such that
\begin{equation}\label{PerPoincare}
\|u\|_{L^{2}(U)} \leq \Theta\,|U|^{1/n}\,\|\nabla u \|_{L^{2}(U)},
\qquad\text{for every $u \in H^{1}_{0}(U)$}.
\end{equation}
\begin{proof}
We first prove \eqref{PerPoincare} for a function $v \in C^{\infty}_{0}(U)$ 
(not identically vanishing on $U$).
Since, in particular, we can think of $v$ as a function belonging to
$C^{\infty}_{0}(\R^n)$, by applying
the Nash inequality (see, e.g., \cite{Nash}) and H\"{o}lder's inequality we get
\begin{equation*}
\begin{aligned}
\|v\|_{L^{2}(U)}^{1+2/n}&=\|v\|_{L^{2}(\R^n)}^{1+2/n} \leq 
\Theta\,\|v\|_{L^{1}(\R^n)}\,\|\nabla v \|_{L^{2}(\R^n)}\\[0.2cm]
& \leq \Theta\,|U|^{1/n}\,\|v\|_{L^{2}(U)}^{2/n}\,\|\nabla v \|_{L^{2}(U)},
\end{aligned}
\end{equation*}
where $\Theta > 0$ is a real constant only depending on the dimension $n$.
As a consequence, since we have assumed that $v\not\equiv 0$ on $U$, we obtain
$$\|v\|_{L^{2}(U)} \leq \Theta\,|U|^{1/n}\,\|\nabla v \|_{L^{2}(U)}.$$
The proof of \eqref{PerPoincare} for a general
$u\in H^1_0(U)$ follows by a density argument.
\end{proof}
\end{lemma} 

 \section{Proof of Theorem \ref{thm.MAIN}} \label{sec.proofThmMain}
 In the present section we give the proof
 of our Theorem \ref{thm.MAIN}. In doing this,
 we take for granted all the notations introduced in the preceding sections.
 \begin{proof} [Proof (of Theorem \ref{thm.MAIN}).]
 For every $\lambda\in (\mathbf{a}_\Omega,0)$, we consider the functions
 $$w_i^{(\lambda)} := u_i-u_i^{(\lambda)}, \qquad
 W_\lambda := (w_1^{(\lambda)},\ldots,w_n^{(\lambda)}) = U-U_\lambda.$$
 Taking into account the regularity of $U$ and of $U_\lambda$ (see, respectively,
 Definition \ref{eq.defisolmain} and \eqref{eq.regulUlambda}),
 and reminding that $\Sigma_\lambda\subseteq\Omega\cap\Omega_\lambda$,
 it is readily seen that (for any $0<\alpha<1$)
 \begin{equation} \label{eq.regulWlambda}
  W_\lambda\in C^{1,\alpha}(\Sigma_\lambda\setminus R_\lambda(\Gamma);\R^m)
 \cap C(\overline{\Sigma_\lambda}\setminus R_\lambda(\Gamma);\R^m).
 \end{equation}
 Furthermore, since $U$ solves \eqref{eq.mainsystem}
 and $U_\lambda$ solves \eqref{eq.solveduilambda} we have
 (note that, as $\Omega$ is convex, the reflection of
 $\de\Sigma_\lambda\setminus\mathit{\Pi}_\lambda$ with respect to~$\mathit{\Pi}_\lambda$
  is entirely contained in $\Omega$)
 \begin{equation} \label{eq.solvedbyWlambda}
  \begin{cases}
  -\Delta w_i^{(\lambda)} = \displaystyle\sum_{j = 1}^mc_{ij}(x;\lambda)w_j^{(\lambda)}, &\text{on
  $\Sigma_\lambda\setminus R_\lambda(\Gamma)$}, \\[0.15cm]
  w_i^{(\lambda)} < 0, & 
  \text{on $(\de\Sigma_\lambda\setminus\mathit{\Pi}_\lambda)
  \setminus R_\lambda(\Gamma)$},  \\[0.15cm]
  w_i^{(\lambda)}\equiv 0, & \text{on $\de\Sigma_\lambda\cap\mathit{\Pi}_\lambda$},
  \end{cases}
 \end{equation}
 where $c_{i1}(\cdot;\lambda),\ldots,c_{im}(\cdot; \lambda):
 \Sigma_\lambda\setminus R_\lambda(\Gamma)\to\R$ are defined as follows:
 \begin{equation} \label{eq.deficij}
  c_{ij}(x;\lambda) := \begin{cases}
  \displaystyle \frac{f_i\big(U(x)\big)-f_i\big(U_\lambda(x)\big)}
 {u_j(x)-u_j^{(\lambda)}(x)}, & \text{if $u_i(x)\neq u_i^{(\lambda)}(x)$}, \\[0.2cm]
 0, & \text{otherwise}.
 \end{cases}
 \end{equation}
 As for the case of $U$ and $U_\lambda$,
 since $W_\lambda$ is not of class $C^2$ on $\Sigma_\lambda$, by saying that
 $w_1^{(\lambda)},\ldots,w_m^{(\lambda)}$ solve the system of PDEs in 
 \eqref{eq.solvedbyWlambda} we mean, precisely, that
 \begin{equation} \label{eq.explicitWlambdasolvesPDE}
  \int_{\Sigma_\lambda}\langle\nabla w_i^{(\lambda)},\nabla\varphi\rangle\,\d x
  = \sum_{j = 1}^m\int_{\Sigma_\lambda}c_{ij}(\cdot; \lambda)w_j^{(\lambda)}\,\varphi\,\d x,
  \quad \forall\,\,\varphi\in C_0^\infty(\Sigma_\lambda\setminus R_\lambda(\Gamma),\R).
 \end{equation}
 Moreover, on account of assumption (H.3), we see that 
 \begin{itemize}
  \item[(i)] $c_{ij}(\cdot; \lambda) \geq 0$ for every $i\in\{1,\ldots,m\}$ and every
  $j\neq i$;
  \item[(ii)] there exists a real constant $\mathbf{c}_f > 0$
   such that
   \begin{equation} \label{eq.cijboundedabove}
    |c_{ij}(\cdot; \lambda)| \leq \mathbf{c}_f, \qquad\text{for every
    $i,j\in\{1,\ldots,m\}$ and every $\lambda \in (\mathbf{a}_\Omega,0)$}.
   \end{equation}
 \end{itemize}
 According to the well-established moving planes technique, we now define
 \begin{equation} \label{eq.defsetIlambda0}
 \begin{split}
  \mathcal{I} := \big\{
  & \lambda\in(\mathbf{a}_\Omega,0):\,\text{$w_i^{(t)} < 0$ on
  $\Sigma_t\setminus R_t(\Gamma)\,\,\forall\,\,t\in (\mathbf{a}_\Omega,\lambda)$
  and $\forall\,\,i\in\{1,\ldots,m\}$}\big\}, \\[0.15cm]
  & \qquad\qquad\qquad\text{and}\qquad \lambda_0 := \sup\mathcal{I}.
  \end{split}
 \end{equation}
 Our aim is to demonstrate 
 that $\mathcal{I}\neq \varnothing$ and that $\lambda_0 = 0$. From now on,
 in order to ease the readability, we split the proof into some different steps. \medskip
 
 \textsc{Step I:} In this step we prove that $\mathcal{I}\neq\varnothing$ and that
 $\lambda_0 > \mathbf{a}_\Omega$. 
We fix $t_0 \in (\mathbf{a}_\Omega,0)$ such that $R_{t_0}(\Gamma) \subset \Omega^{c}$. Necessarily,
we have that $R_{t}(\Gamma) \subset \Omega^{c}$ for every $t \in (\mathbf{a}_{\Omega},t_{0})$.
Now, for every $i=1,\ldots,m$ we consider the function $\varphi_{i,h}:\Omega \to \mathbb{R}$
defined as
$$\varphi_{i,h} := (w_{i}^{(t)})^{+} \phi^{2}_{h} \mathbf{1}_{\Sigma_{t}},$$
\noindent where $\{\phi_{h}\}_{h\in \mathbb{N}}$ is as in \eqref{eq.defiphihT}. By density, we can use $\varphi_{i,h}$ as 
a test function, finding
\begin{equation*}
\dfrac{1}{2}\int_{\Sigma_{t}} |\nabla (w_{i}^{(t)})^{+}|^2 \phi_{h}^2 \,\d x 
\leq \mathbf{c}_f \sum_{j=1}^{m}\int_{\Sigma_{t}}(w_{j}^{(t)})^{+} (w_{i}^{(t)})^{+} \phi_h^2 \,\d x + 2 \int_{\Sigma_{t}}(w_{i}^{(t)})^2 |\nabla \phi_{h}|^2 \,\d x.
\end{equation*}
By Fatou Lemma, sending $h \to 0^+$ we get
\begin{equation*}
\dfrac{1}{2}\int_{\Sigma_{t}} |\nabla (w_{i}^{(t)})^{+}|^2 \, dx \leq \mathbf{c}_f \sum_{j=1}^{m}\int_{\Sigma_{t}}(w_{j}^{(t)})^{+} (w_{i}^{(t)})^{+} \,\d x.
\end{equation*}
By H\"{o}lder inequality on every term on the right hand side, we get
\begin{equation*}
\dfrac{1}{2}\|\nabla (w_{i}^{(t)})^{+}\|_{L^{2}(\Sigma_{t})}^2 \,\d x 
\leq \mathbf{c}_f \sum_{j=1}^{m} \|(w_{j}^{(t)})^{+}\|_{L^{2}(\Sigma_t)} 
\|(w_{i}^{(t)})^{+}\|_{L^{2}(\Sigma_t)}.
\end{equation*}
{F}rom this, by using \eqref{PerPoincare} (on every term on the right hand side),
for every 
$t\in (\mathbf{a}_\Omega,t_0)$ and every
index $i\in\{1,\ldots,m\}$ we get
\begin{equation*}
\dfrac{1}{2}\|\nabla (w_{i}^{(t)})^{+}\|_{L^{2}(\Sigma_{t})}^2 \,\d x \leq 
\mathbf{c}_f\,\theta_n^{2}(\Sigma_{t})\,\sum_{j=1}^{m} 
\|\nabla (w_{j}^{(t)})^{+}\|_{L^{2}(\Sigma_t)} 
\|\nabla (w_{i}^{(t)})^{+}\|_{L^{2}(\Sigma_t)},
\end{equation*}
\noindent where we have introduced the notation
 (repeatedly used in the sequel)
 $$\theta_n (\Sigma_{t}) := \Theta\,|\Sigma_{t}|^{1/n}
 \quad \text{(with $\Theta > 0$ is as in Lemma \ref{LemmaPoincare}).}$$ 
 Now, if $\|\nabla (w_{i}^{(t)})^{+}\|_{L^{2}(\Sigma_{t})}\neq 0$,
 from the above inequality we immediately get
 \begin{equation}\label{eq.BaseInduzione}
 \dfrac{1}{2}\|\nabla (w_{i}^{(t)})^{+}\|_{L^{2}(\Sigma_{t})}
 \leq \mathbf{c}_f\,\theta_n^{2}(\Sigma_{t})\,\sum_{j=1}^{m} 
 \|\nabla (w_{j}^{(t)})^{+}\|_{L^{2}(\Sigma_t)}.
 \end{equation}
 On the other hand, since \eqref{eq.BaseInduzione}
 is trivially satisfied when $\|\nabla (w_{i}^{(t)})^{+}\|_{L^{2}(\Sigma_{t})} = 0$,
 we conclude that such an inequality holds true
 for every $i\in\{1,\ldots,m\}$ and every $t\in(\mathbf{a}_\Omega,t_0)$. 
 
 We now aim at proving the following
 assertion: \emph{for every fixed
 $k\in\{1,\ldots,m-1\}$ there exist $t_k \in (\mathbf{a}_{\Omega},t_0)$ and a real constant
 $C_k=C_k(m,\mathbf{c}_f) > 0$  such that}
 \begin{equation}\label{eq:Claim}
 \|\nabla (w_{i}^{(t)})^+\|_{L^{2}(\Sigma_t)} \leq 
 C_k\,\theta^{2}_{n}(\Sigma_{t}) \sum_{j\geq i+1} \|\nabla 
 (w_{j}^{(t)})^+\|_{L^{2}(\Sigma_t)}, \quad
 \begin{array}{c}
 \text{for all $1\leq i\leq k$ and} \\[0.1cm]
 \text{every $t\in(\mathbf{a}_\Omega,t_k)$}.
 \end{array}
\end{equation}
To prove \eqref{eq:Claim} we argue by 
(finite) induction and we start with $k=1$.
By \eqref{eq.BaseInduzione} we have
\begin{equation*}
\left(\frac{1}{2}- \mathbf{c}_f\,\theta^2_{n}(\Sigma_t)\right)
\|\nabla (w_{1}^{(t)})^{+}\|_{L^{2}
(\Sigma_{t})}
\leq \mathbf{c}_f\,\theta_n^2(\Sigma_{t})\,\sum_{j\geq 2} 
\|\nabla (w_{j}^{(t)})^{+}\|_{L^{2}(\Sigma_t)}.
\end{equation*}
Since $\theta_n(\Sigma_t) \to 0$ as $t \to \mathbf{a}_{\Omega}$, it is possible
to find $t_1 \in (\mathbf{a}_{\Omega}, t_0)$ such that 
$$\displaystyle\frac{1}{2}- \mathbf{c}_f\,\theta^2_{n}(\Sigma_t) > \frac{1}{4}.\quad\text{for
every $t\in(\mathbf{a}_\Omega,t_1)$}.$$
As a consequence, we obtain
\begin{equation*}
\|\nabla (w_{1}^{(t)})^{+}\|_{L^{2}(\Sigma_{t})}\leq 4 \mathbf{c}_{0}\,
\theta^{2}_{n}(\Sigma_t)\,\sum_{j\geq 2} \|\nabla (w_{j}^{(t)})^{+}\|_{L^{2}(\Sigma_t)},
\end{equation*}
 which is precisely \eqref{eq:Claim} for $i=1$
(with $C_1 = 4\mathbf{c}_0$).
 Let us now suppose that \eqref{eq:Claim} holds for a certain index 
 $k \in \{1, \ldots, m-2\}$ and,
 by shrinking $t_k$ if necessary, let us also assume that $\theta_n(\Sigma_t) < 1$
 for all $t\in(\mathbf{a}_\Omega,t_k)$.
 Owing to \eqref{eq.BaseInduzione} (with $i=k+1$), we then have 
\begin{equation}\label{eq.StimaWk+1}
\begin{split}
 &\left(\frac{1}{2}- \mathbf{c}_f\,\theta^2_{n}(\Sigma_t)\right)
 \|\nabla (w_{k+1}^{(t)})^{+}\|_{L^{2}(\Sigma_{t})} 
 \leq \mathbf{c}_f\,\theta_n^{2}(\Sigma_{t})
 \sum_{i\neq k+1} \|\nabla (w_{i}^{(t)})^{+}\|_{L^{2}(\Sigma_t)} \\[0.2cm]
 & \quad = \mathbf{c}_f\,\theta^{2}_{n}(\Sigma_t)
 \bigg(\sum_{i = 1}^k\|\nabla (w_{i}^{(t)})^{+}\|_{L^{2}(\Sigma_{t})}
 + 
 \sum_{i\geq k+2}\|\nabla (w_{i}^{(t)})^{+}\|_{L^{2}(\Sigma_{t})}\bigg) \\[0.2cm]
 &\quad \big(\text{by \eqref{eq:Claim},
 which we are assuming to hold for the index $k$}\big)\\[0.2cm]
 &\quad \leq \mathbf{c}_f\,\theta^{2}_{n}(\Sigma_t)\bigg(
 C_k\,\sum_{i = 1}^k\sum_{j \geq i+1}
 \|\nabla (w_{j}^{(t)})^{+}\|_{L^{2}(\Sigma_{t})}
 + \sum_{i\geq k+2}\|\nabla (w_{i}^{(t)})^{+}\|_{L^{2}(\Sigma_{t})}\bigg) \\[0.2cm]
 & \quad \leq
 \mathbf{c}_f\,\theta^{2}_{n}(\Sigma_t)\bigg(
 m\,C_k\,\sum_{j = 2}^{k+1}\|\nabla (w_{j}^{(t)})^{+}\|_{L^{2}(\Sigma_{t})}\,+
 \\[0.2cm] 
 & \quad\qquad\qquad\qquad\qquad
 + (k\,C_k+1)\,\sum_{i\geq k+2}\|\nabla (w_{i}^{(t)})^{+}\|_{L^{2}(\Sigma_{t})}\bigg)
 =: (\star),
\end{split}
\end{equation}
 We now perform a backward induction
 argument to show that, as a consequence of the validity of
 \eqref{eq:Claim} for the index $k$, the following fact holds:
 \emph{for every fixed $j\in\{1,\ldots,k\}$, it is possible to find
 a real constant
 $\mathcal{C}_j = \mathcal{C}_j(m,k,\mathbf{c}_f)>0$ such that}
\begin{equation}\label{eq.Claim2}
\|\nabla (w_{j}^{(t)})^{+}\|_{L^{2}(\Sigma_{t})} \leq 
\mathcal{C}_j\,\theta^2_n (\Sigma_t)
\sum_{r\geq k+1}\|\nabla (w_{r}^{(t)})^{+}\|_{L^{2}(\Sigma_{t})},
\quad
 \begin{array}{c}
 \text{for all $1\leq j\leq k$ and} \\[0.1cm]
 \text{every $t\in(\mathbf{a}_\Omega,t_k)$}.
 \end{array}
\end{equation}
 For $j=k$, \eqref{eq.Claim2} follows immediately from \eqref{eq:Claim} 
 by taking $i=k$ (with $\mathcal{C}_k := C_k$). 
 We then suppose the existence of 
 an index $j\in \{2,\ldots,k\}$ such that
 \eqref{eq.Claim2} holds for every $j\leq r\leq k$,
 and we exploit once again
 \eqref{eq:Claim} (with $i=j-1\leq k-1$): this gives
\begin{equation*}
\begin{split}
& \|\nabla (w_{j-1}^{(t)})^{+}\|_{L^{2}(\Sigma_{t})} \leq 
C_k\,\theta^{2}_{n}(\Sigma_t) \sum_{r \geq j} \|\nabla (w_{r}^{(t)})^{+}\|_{L^{2}(\Sigma_t)}
\\
& \qquad\quad =  C_k\,\theta^{2}_{n}(\Sigma_t)\bigg(
\sum_{r = j}^{k}\|\nabla (w_{r}^{(t)})^{+}\|_{L^{2}(\Sigma_t)}
+\sum_{r \geq k+1} \|\nabla (w_{r}^{(t)})^{+}\|_{L^{2}(\Sigma_t)}\bigg)\\[0.2cm]
& \qquad\quad \big(\text{since \eqref{eq.Claim2} holds for $j\leq r\leq k$, and
$\theta_n(\Sigma_t)<1$}\big) \\[0.2cm]
& \qquad\quad \leq
C_k\,\big(m\max_{j\leq r\leq k}(C_r)+1\big)\,\theta_n^2(\Sigma_t)\,
\sum_{r\geq k+1}\|\nabla (w_{r}^{(t)})^{+}\|_{L^{2}(\Sigma_{t})},
\end{split}
\end{equation*} 
\noindent so that \eqref{eq.Claim2} holds true also for $j-1$.
By the Induction Principle, we then conclude that
estimate \eqref{eq.Claim2} is valid for every $j = 1,\ldots,k$, as claimed. \vspace*{0.1cm} 

With \eqref{eq.Claim2} at hand, we now continue
the estimate \eqref{eq.StimaWk+1}: reminding that,
by the choice of $t_k$, we have
$\theta_n(\Sigma_t) < 1$ for every $t\in(\mathbf{a}_\Omega,t_k)$, we have
\begin{equation}\label{eq.RHS2}
(\star) \leq M_k\,\theta_n^2(\Sigma_t)
\sum_{j\geq k+1}\|\nabla (w_{j}^{(t)})^{+}\|_{L^{2}(\Sigma_{t})},
\end{equation}
 where $M_k = M_k(m,\mathbf{c}_f) > 0$ is a suitable
As a consequence, we obtain
\begin{equation*}
\left(\frac{1}{2}-\mathbf{c}_f 
\theta_n^2(\Sigma_t) - M_k\,\theta_n^2(\Sigma_t)\right)\|\nabla(w_{k+1}^{(t)})^{+}\|_{L^{2}(\Sigma_t)} 
\leq  M_k\,\theta_n^2(\Sigma_t)\sum_{j\geq k+2}\|\nabla (w_{j}^{(t)})^{+}\|_{L^{2}(\Sigma_{t})}.
\end{equation*}
Finally, since $\theta_n(\Sigma_t)\to 0$ as $t \to \mathbf{a}_{\Omega}$, we infer the existence
 of $\bar{t}\in (\mathbf{a}_{\Omega},t_0)$ such that
 $$\frac{1}{2}-\mathbf{c}_f 
\theta_n^2(\Sigma_t) - M_k\,\theta_n^2(\Sigma_t) > 
 \frac{1}{4} \quad \text{for every $t \in (\mathbf{a}_\Omega, \bar{t}$)};$$
 from this, we obviously derive the estimate (valid for $t\in(\mathbf{a}_\Omega,\bar{t})$)
$$\|\nabla(w_{k+1}^{(t)})^{+}\|_{L^{2}(\Sigma_t)} 
\leq  4M_k\,\theta_n^2(\Sigma_t)\sum_{j\geq k+2}\|\nabla (w_{j}^{(t)})^{+}\|_{L^{2}(\Sigma_{t})}.$$
Taking as $t_{k+1}:= \min \{t_k,\bar{t}\}$, 
and setting $C_{k+1} := \max\{C_k,4M_k\}$,
we then obtain
$$\|\nabla(w_{i}^{(t)})^{+}\|_{L^{2}(\Sigma_t)} 
\leq  C_{k+1}\,\theta_n^2(\Sigma_t)\sum_{j\geq i+1}\|\nabla (w_{i}^{(t)})^{+}\|_{L^{2}(\Sigma_{t})}
\quad 
 \begin{array}{c}
 \text{for all $1\leq i\leq k+1$ and} \\[0.1cm]
 \text{every $t\in(\mathbf{a}_\Omega,t_{k+1})$},
 \end{array}$$
 so that \eqref{eq:Claim} holds true also for $k+1$.  By the Induction Principle, we conclude that
estimate \eqref{eq:Claim} is valid for every $k = 1,\ldots,m-2$, as claimed. \vspace*{0.1cm} 

 Now we have established \eqref{eq:Claim}, we are able to complete the proof 
 this step. In fact, since the cited \eqref{eq:Claim} holds true for $k = m-1$, a (finite)
 backward induction argument shows the existence
 of a real constant $\mathcal{C}_m = \mathcal{C}_m(\mathbf{c}_0) > 0$ such that
 \begin{equation}\label{eq.Claim2finale}
  \|\nabla (w_{j}^{(t)})^{+}\|_{L^{2}(\Sigma_{t})} \leq 
  \mathcal{C}_m\,\theta^2_n (\Sigma_t)\,\|\nabla (w_{m}^{(t)})^{+}\|_{L^{2}(\Sigma_{t})},
   \quad 
 \begin{array}{c}
 \text{for all $1\leq j\leq m-1$ and} \\[0.1cm]
 \text{every $t\in(\mathbf{a}_\Omega,t_{m-1})$};
 \end{array}
\end{equation}
  gathering together
 \eqref{eq.Claim2finale} and \eqref{eq.BaseInduzione}
 (with $i = m$),
 for any $t\in(\mathbf{a}_\Omega,t_{m-1})$ we get
 $$\dfrac{1}{2}\|\nabla (w_{m}^{(t)})^{+}\|_{L^{2}(\Sigma_{t})}\leq
 \mathbf{c}_f\,\theta_n^{2}(\Sigma_{t})\,(m\,\mathcal{C}_m\,\theta^2_n(\Sigma_t)+1)\,
 \|\nabla (w_{m}^{(t)})^{+}\|_{L^{2}(\Sigma_t)}.$$
 Since $\theta_n(\Sigma_t)\to 0$ as $t\to \mathbf{a}_\Omega$, there exists
 $\tau_0\in(\mathbf{a}_\Omega,t_{m-1})$ such that
 $$\mathbf{c}_f\,\theta_n^{2}(\Sigma_{t})\,(m\,\mathcal{C}_m\,\theta^2_n(\Sigma_t)+1)
 < \frac{1}{4},\quad \text{for every $t\in(\mathbf{a}_\Omega,\tau_0)$};$$
 as a consequence, we obtain
 $$\|\nabla (w_m^{(t)})\|_{L^{2}(\Sigma_t)} = 0,\quad
 \text{for every $t\in(\mathbf{a}_\Omega,\tau_0)$}.$$
 On account of \eqref{eq:Claim}, this proves that 
 \begin{equation*} 
  \|\nabla (w_1^{(t)})\|_{L^{2}(\Sigma_t)} = \cdots
 = \|\nabla (w_m^{(t)})\|_{L^{2}(\Sigma_t)} = 0,\quad
 \text{for every $t\in(\mathbf{a}_\Omega,\tau_0)$};
 \end{equation*}
 as a consequence, by Lemma \ref{LemmaPoincare} 
 (and since $W_t$ is continuous on $\Sigma_t$) we get 
 \begin{equation} \label{eq.daRichiamareinStepII}
  \text{$u_{i} - u_{i}^{(t)} = w_i^{(t)}\leq 0$ 
  on $\Sigma_t$ \quad (for every $i = 1,\ldots,m$ and every
 $t\in(\mathbf{a}_\Omega,\tau_0)$).}
 \end{equation} 
 We finally claim that,
 by the Strong Maximum Principle for $C^1$-subsolutions, we have
 \begin{equation} \label{eqwistrictlynegativeStepI}
  \text{$u_i < u_i^{(t)}$ on $\Sigma_t$}, \quad \text{for every
  $i = 1,\ldots,m$ and every $t\in (\mathbf{a}_\Omega,\tau_0)$}.
 \end{equation}
  Indeed, let $i\in\{1,\ldots,m\}$ and $t\in(\mathbf{a}_\Omega,\tau_0)$
  be arbitrarily fixed. Clearly,
 the set $\Sigma_t$ is
 (open and) connected;
 moreover, since the (vector-valued) map $W_t = U-U_t$
 solves \eqref{eq.solvedbyWlambda} and 
 $c_{ij}(\cdot; t)\geq 0$ for every $j\neq i$,
 we have 
 \begin{align*}
  -\Delta w_i^{(t)} = \sum_{j = 1}^m c_{ij}(\cdot; t)w_j^{(t)}
  \leq c_{ii}(\cdot; t)\,w_i^{(t)}\quad( \text{as $u_i\leq u_i^{(t)}$} ).
 \end{align*}
 We explicitly point out that the above inequality has to be intended
 in the weak sense of distributions
 on $\Sigma_t$: this means, precisely, that
 $$\int_{\Sigma_{t}}\langle\nabla w_i^{(t)},\nabla\varphi\rangle\,\d x
 \leq \int_{\Sigma_{t}}c_{ii}(\cdot;t)\,w_i^{(t)}\,\d x,
 \quad\text{$\forall\,\,\varphi\in C_0^\infty(\Sigma_{t},\R)$ with $\varphi\geq 0$ on
 $\Sigma_t$.}$$
 From this, taking into account \eqref{eq.cijboundedabove} we get
 \begin{equation*}
  -\Delta w_i^{(t)} + \big(\mathbf{c}_f-c_{ii}(\cdot; t)\big)w_i^{(t)}
 \leq \mathbf{c}_fw_i^{(t)}\leq 0,
 \end{equation*}
 and $\mathbf{c}_f-c_{ii}(\cdot; t)\geq 0$ on $\Sigma_{t}$.
 Gathering together all these facts, 
 we can invoke the Strong Maximum Principle
  for $C^1$-subsolution (see, e.g., \cite{GT}), ensuring that
  $$\text{either $w_i^{(t)} < 0$ or
  $w_i^{(t)}\equiv 0$ on $\Sigma_{t}$}.$$
  Since, by \eqref{eq.solvedbyWlambda}, we know that
 the function $w_{t}^{(t)}$ is (strictly) negative
 on the set $\de\Sigma_{t}\setminus\mathit{\Pi}_{t}$
 (notice $t < \tau_0 < 0$),
 we then conclude that \eqref{eqwistrictlynegativeStepI} holds true. \vspace*{0.12cm}
 
 Finally, on account of \eqref{eqwistrictlynegativeStepI} (and taking into account
 the very definition of $\mathcal{I}$), we see that $(\mathbf{a}_\Omega,\tau_0)\subseteq\mathcal{I}$,
 whence $\mathcal{I}\neq\varnothing$, and 
 that $\lambda_0 = \sup\mathcal{I} \geq \tau_0 > \mathbf{a}_\Omega$.
\medskip

 \textsc{Step II:} We now turn to demonstrate that $\lambda_0 = 0$. To this end,
 following \cite{EFS}, we argue by contradiction and we assume that
 $\lambda_0 \in (\mathbf{a}_\Omega,0)$.
 Since $W_{\lambda_0}$ is continuous on $\Sigma_{\lambda_0}\setminus R_{\lambda_0}(\Gamma)$,
 from the very definition of $\lambda_0$ we deduce that,
 for every $i\in\{1,\ldots,m\}$,
 \begin{equation} \label{eq.winonpositivelambda0}
  \text{$w_i^{(\lambda_0)}\leq 0$ on $\Sigma_{\lambda_0}\setminus R_{\lambda_0}(\Gamma)$, that is,
 $u_i\leq u_{\lambda_0}$ on $\Sigma_{\lambda_0}\setminus R_{\lambda_0}(\Gamma)$.}
 \end{equation}
 As a consequence, by the Strong Maximum Principle
 (for $C^1$-subsolutions) we get
 \begin{equation} \label{eq.wistrctlynegative}
  \text{$w_i^{(\lambda_0)} < 0$ on $\Sigma_{\lambda_0}\setminus R_{\lambda_0}(\Gamma)$}, 
  \qquad\text{for every $i\in\{1,\ldots,m\}$}.
 \end{equation}
 In fact, taking into account that
 $W_{\lambda_0}$ solves 
 \eqref{eq.solvedbyWlambda}
 and arguing exactly as in the last part of the previous step, we have
 the following family of inequalities (which has to be indended
 in the weak sense of distributions on 
 $\Sigma_{\lambda_0}\setminus R_{\lambda_0}(\Gamma)$):
 \begin{equation} \label{eq.solvedbywilambdaperHopf}
  -\Delta w_i^{(\lambda_0)} + \big(\mathbf{c}_f-c_{ii}(\cdot;\lambda_0)\big)w_i^{(\lambda_0)}
 \leq \mathbf{c}_fw_i^{(\lambda_0)}\leq 0 \quad \text{for every $i\in\{1,\ldots,m\}$}.
 \end{equation}
 Moreover, since $\mathbf{c}_f-c_{ii}(\cdot;\lambda_0)\geq 0$ 
 on $\Sigma_{\lambda_0}\setminus R_{\lambda_0}(\Gamma)$
 (see \eqref{eq.cijboundedabove}) and since, by 
 Lemma \ref{lem.topologico}-(1),
 the set $\Sigma_{\lambda_0}\setminus R_{\lambda_0}(\Gamma)$ is
 open and connected (see also Remark \ref{rem.casofacile}), 
 we are entitled to apply the Strong Maximum Principle
  for $C^1$-subsolution: this gives
  $$\text{either $w_i^{(\lambda_0)} < 0$ or
  $w_i^{(\lambda_0)}\equiv 0$ on $\Sigma_{\lambda_0}\setminus R_{\lambda_0}(\Gamma)$}\qquad
  (\text{for any $i\in\{1,\ldots,m\}$}).$$
  Finally, since we know that
 the functions $w_{1}^{(\lambda_0)},\ldots,w_{m}^{(\lambda_0)}$ are (strictly) negative
 on the set $(\de\Sigma_{\lambda_0}\setminus\mathit{\Pi}_{\lambda_0})\setminus R_{\lambda_0}(\Gamma)$
 (as $\lambda_0 < 0$, see \eqref{eq.solvedbyWlambda}),
 we conclude that \eqref{eq.wistrctlynegative} holds true. \medskip
 
 Now we have established \eqref{eq.wistrctlynegative},
 we then turn to prove the following 
 assertion:
 \emph{in cor\-re\-spon\-den\-ce 
 to every compact set $K\subseteq \Sigma_{\lambda_0}\setminus R_{\lambda_0}(\Gamma)$
 with Lipschitz boundary $\de K$,
 it is possible to find a small $\epsilon = \epsilon(K,\lambda_0) \in (0,|\lambda_0|/2)$ 
 such that}
 \begin{itemize}
  \item[(a)] $K\subseteq \Sigma_\lambda\setminus R_\lambda(\Gamma)$ 
  for every
  $\lambda\in [\lambda_0,\lambda_0+\epsilon]$;
  \item[(b)] $(w_i^{(\lambda)})^+\equiv 0$ on $K$ for every $i\in\{1,\ldots,m\}$
  and every $\lambda\in(\lambda_0,\lambda_0+\epsilon]$;
  \item[(c)] for every
  $i\in\{1,\ldots,m\}$ and every $\lambda\in(\lambda_0,\lambda_0+\epsilon]$
  we have \phantom{$w_i^{(\lambda)}$}
  \begin{equation} \label{eq.toprovebeforechooseK}
   \|\nabla (w_i^{(\lambda)})^+\|_{L^2(\Sigma_\lambda\setminus K)}
   \leq \mathbf{c}_f\,\theta^2_n\big(\Sigma_\lambda\setminus K\big)\,\sum_{j = 1}^m
   \|\nabla (w_j^{(\lambda)})^+\|_{L^2(\Sigma_\lambda\setminus K)},
  \end{equation}
  where
  $\theta_n\big(\Sigma_\lambda\setminus K\big)
  = \Theta\,\big|\Sigma_\lambda\setminus K\big|^{1/n}$ (see Lemma \ref{LemmaPoincare}).
 \end{itemize}
 We explicitly observe that, if $\epsilon<|\lambda_0|/2$, we have
 $$[\lambda_0,\lambda_0+\epsilon]\subseteq (\mathbf{a}_\Omega,0).$$
 Let now $K\subseteq\Sigma_{\lambda_0}\setminus R_{\lambda_0}(\Gamma)$
  be an ar\-bi\-tra\-ri\-ly fixed compact set. Since both $K$ and $R_{\lambda_0}(\Gamma)$ are closed,
  it is very easy to recognize that
  there exists a suitable $\nu = \nu(K,\lambda_0) > 0$, which we can assume
  to be smaller than $|\lambda_0|/2$, such that
  \begin{equation} \label{eq.KinSigmalambdanu}
  K\subseteq \Sigma_\lambda\setminus R_\lambda(\Gamma), \qquad\text{for
  every $\lambda\in[\lambda_0,\lambda_0+\nu]$}.
  \end{equation}
  Moreover,
  on account of \eqref{eq.wistrctlynegative} (and remembering
  that $W_{\lambda_0}$ is continuous on $\Sigma_{\lambda_0}\setminus R_{\lambda_0}(\Gamma)$),
  it is possible to find a real constant $M_0 < 0$ such that
  \begin{equation} \label{eq.wilessM0K}
  \text{$w_i^{(\lambda_0)}\leq M_0 < 0$ on 
  $\Sigma_{\lambda_0}\setminus R_{\lambda_0}(\Gamma)$}, \qquad\text{for
  every $i\in\{1,\ldots,m\}$}.
  \end{equation}
  Since, for every fixed $i\in\{1,\ldots,m\}$ and every
  $\lambda\in[\lambda_0,\lambda_0+\nu]$, the function
  $$
  (x,\lambda)\mapsto w_i^{(\lambda)}(x) 
  = u_i(x)-u_i^{(\lambda)}(x),$$
  is (well-defined and) uniformly continuous on $K\times[\lambda_0,\lambda_0+\nu]$ (as it follows
  from \eqref{eq.KinSigmalambdanu}), there exists
  a real $\epsilon = \epsilon(K,\lambda_0) \in (0,\nu)$ (hence, $\epsilon < |\lambda_0|/2$)
  such that
  \begin{equation} \label{eq.wilambdanegativegeneral}
  w_i^{(\lambda)}(x) < w_i^{(\lambda_0)}(x)+\frac{|M_0|}{2}
  \stackrel{\eqref{eq.wilessM0K}}{\leq} \frac{M_0}{2} < 0, \quad\text{$\forall\,\,x\in K$
  and $\forall\,\,\lambda\in[\lambda_0,\lambda_0+\epsilon]$}.
  \end{equation}
  Summing up, if $\lambda\in[\lambda_0,\lambda_0+\epsilon]$, we have $K\subseteq
  \Sigma_\lambda\setminus R_\lambda(\Gamma)$ and 
  $(w_i^{(\lambda)})^+\equiv 0$ on $K$. 
  
  We then turn to prove \eqref{eq.toprovebeforechooseK}. 
  To this end, let $i\in\{1,\ldots,m\}$ 
  and let $\lambda\in(\lambda_0,\lambda_0+\epsilon]$
  be arbitrarily fixed. We consider the (double) sequence of functions defined by
  $$\varphi_{h,k} := \begin{cases}
 (w_i^{(\lambda)})^+\,\phi^2_h\,\psi^2_k, &\text{on $\in\Sigma_\lambda$}, \\
 0, & \text{on $\R^n\setminus\Sigma_\lambda$},
 \end{cases}$$
  where $\{\phi_h\}_{h\in\N}$ is the sequence 
  defined in \eqref{eq.defiphihT} and associated with 
  $\gamma_{\lambda} = \de\Omega\cap\mathit{\Pi}_{\lambda}$, whilst
  $\{\psi_k\}_{k\in\N}$ is the sequence 
  defined in \eqref{eq.defipsikT} and associated with 
  $R_{\lambda}(\Gamma)$
 (actually, the functions $\varphi_{h,k}$ also depend on
  the fixed
  $i$ and $\lambda$; however, in order to avoid cumbersome
  notations, we prefer to not keep
  trace of this dependence in the sequel). \vspace*{0.02cm}
  
  By Lemma \ref{lem.cutoffRlambda},
  for every $h,k\in\N$ we have $\varphi_{h,k}\in \Lip(\overline{\Sigma_\lambda})$
  and $\varphi_{h,k}\equiv 0$ on $\de\Sigma_\lambda$; moreover, by
  \eqref{eq.wilambdanegativegeneral}, there exists an open neighborhood
  $\mathcal{U}\subseteq\Sigma_\lambda\setminus R_\lambda(\Gamma)$ of
  $K$ such that
  \begin{equation} \label{eq.varphihk0onU}
  \text{$(w_i^{(\lambda)})^+ \equiv 0$
  on $\mathcal{U}$}, \quad 
  \text{whence $\varphi_{h,k}\equiv 0$ on $\mathcal{U}$ for every $h,k\in\N$}.
  \end{equation}
  Gathering together all these facts, we deduce that 
  \begin{equation} \label{eq.varphihkH0outK}
  \varphi_{h,k}\in H_0^1(\Sigma_\lambda\setminus K),
  \end{equation}
  Furthermore, since $\varphi_{h,k}\to (w_i^{(\lambda)})^+$ in $H_0^1(\Sigma_\lambda)$ 
  as $h,\,k\to\infty$ (see Lemma \ref{lem.regularitywlambda}), we also get
  \begin{equation} \label{eq.wilambdaposoutK}
   (w_i^{(\lambda)})^+\in H_0^1(\Sigma_\lambda\setminus K).
  \end{equation}
  Owing to \eqref{eq.varphihkH0outK}, and by a standard
  density argument, we are entitled to
  use the function
  $\varphi_{h,k}$ (for every fixed $h,k\in\N$) 
  as test function in \eqref{eq.explicitWlambdasolvesPDE}, obtaining
  (see also \eqref{eq.nablavarphihtwo})
  \begin{align*}
   & \int_{\Sigma_\lambda}|\nabla 
   (w_i^{(\lambda)})^+|^2\,\phi_h^2\,\psi_k^2\,\d x 
   + 2\int_{\Sigma_\lambda}(w_i^{(\lambda)})^+\,
   \psi^2_k\,\phi_h\,\langle\nabla w_i^{(\lambda)},
   \nabla\phi_h\rangle\,\d x\,+
   \\[0.2cm]
   & \qquad\qquad\qquad + 2\int_{\Sigma_\lambda}(w_i^{(\lambda)})^+\,
   \phi^2_h\,\psi_k\,\langle \nabla w_i^{(\lambda)},\nabla\psi_k\rangle\,\d x \\[0.2cm]
   & \qquad
   = \int_{\Sigma_\lambda}\langle\nabla w_i^{(\lambda)},\nabla\varphi_{h,k}\rangle\,\d x
  = \sum_{j = 1}^m\int_{\Sigma_\lambda}c_{ij}(\cdot; \lambda)\,w_j^{(\lambda)}
  \,\varphi_{h,k}\,\d x \\[0.2cm]
  & \qquad
  = \sum_{j = 1}^m\int_{\Sigma_\lambda}c_{ij}(\cdot; \lambda)\,(w_i^{(\lambda)})^+ w_j^{(\lambda)}
  \,\phi_h^2\,\psi_k^2\,\d x.
  \end{align*}
  From this, by \eqref{eq.varphihk0onU}, \eqref{eq.cijboundedabove} 
  and the fact
  that $c_{ij}(\cdot;\lambda)\geq 0$ if $j\neq i$, we get
  \begin{align*}
   & \int_{\Sigma_\lambda\setminus K}|\nabla 
   (w_i^{(\lambda)})^+|^2\,\phi_h^2\,\psi_k^2\,\d x 
   = \int_{\Sigma_\lambda}|\nabla 
   (w_i^{(\lambda)})^+|^2\,\phi_h^2\,\psi_k^2\,\d x 
   \\[0.2cm]
   & \qquad\leq 2\int_{\Sigma_\lambda\setminus K}(w_i^{(\lambda)})^+\,
   \psi^2_k\,\phi_h\,|\nabla w_i^{(\lambda)}|\,
   |\nabla\phi_h|\,\d x\,+ 2\int_{\Sigma_\lambda\setminus K}(w_i^{(\lambda)})^+\,
   \phi^2_h\,\psi_k\,|\nabla w_i^{(\lambda)}|\,|\nabla\psi_k|\,\d x\,+ \\[0.2cm]
   &\qquad\qquad\qquad + \mathbf{c}_f\sum_{j = 1}^m\int_{\Sigma_\lambda\setminus K}
   (w_i^{(\lambda)})^+ (w_j^{(\lambda)})^+
  \,\phi_h^2\,\psi_k^2\,\d x.
  \end{align*}
  We now observe that, since $\nabla w_i^{(\lambda)} = 
  \nabla (w_i^{(\lambda)})^+$ almost everywhere on the set
  $\{w_i^{(\lambda)} > 0\}$, the above inequality can be re-written as
  follows:
  \begin{equation} \label{eq.primostepnumerato}
  \begin{split}
   & \int_{\Sigma_\lambda\setminus K}|\nabla 
   (w_i^{(\lambda)})^+|^2\,\phi_h^2\,\psi_k^2\,\d x 
   \leq \int_{\Sigma_\lambda\setminus K}2\big((w_i^{(\lambda)})^+\,
   \psi_k\,|\nabla\phi_h|\big)\big(\psi_k\phi_h\,|\nabla (w_i^{(\lambda)})^+|\big)
   \,\d x\,+ \\[0.2cm] 
   & \qquad\qquad + \int_{\Sigma_\lambda\setminus K}2\big((w_i^{(\lambda)})^+\,
   \phi_h\,|\nabla\psi_k|\big)\big(\,\phi_h\psi_k\,|\nabla (w_i^{(\lambda)})^+|\big)\,\d x\,+
    \\[0.2cm]
    & \qquad\qquad+ \mathbf{c}_f\sum_{j = 1}^m\int_{\Sigma_\lambda\setminus K}
   (w_i^{(\lambda)})^+ (w_j^{(\lambda)})^+
  \,\phi_h^2\,\psi_k^2\,\d x.
  \end{split}
  \end{equation}
  From this, by using the classical Young's inequality
  $$2 a b\leq 4a^2+\frac{1}{4} b^2 \qquad\big(\text{holding true for every $a,b\geq 0$}\big)$$
  on the integrands of the first two integrals
  in the right-hand side of \eqref{eq.primostepnumerato}, we get
  \begin{align*}
   & \frac{1}{2}\int_{\Sigma_\lambda\setminus K}|\nabla 
   (w_i^{(\lambda)})^+|^2\,\phi_h^2\,\psi_k^2\,\d x \leq
   4\,\int_{\Sigma_\lambda\setminus K}\big[(w_i^{(\lambda)})^+\big]^2\,
   \psi_k^2\,|\nabla\phi_h|^2
   \,\d x\,+ \\[0.2cm] 
   & \qquad+4\int_{\Sigma_\lambda\setminus K}\big[(w_i^{(\lambda)})^+\big]^2
   \phi_h^2\,|\nabla\psi_k|^2\,\d x
   + \mathbf{c}_f\sum_{j = 1}^m\int_{\Sigma_\lambda\setminus K}
   (w_i^{(\lambda)})^+ (w_j^{(\lambda)})^+
  \,\phi_h^2\,\psi_k^2\,\d x.
  \end{align*}
  To proceed further towards the proof of \eqref{eq.toprovebeforechooseK} we observe that,
  since $u_1,\ldots,u_m$ are positive on $\Omega\setminus \Gamma$
  and $R_\lambda(\Sigma_\lambda)\subseteq\Omega$ (by convexity), we have
  $$0\leq (w_i^{(\lambda)})^+ = (u_i-u_i^{(\lambda)})^+\leq u_i, \qquad\text{
  on $\Sigma_\lambda\setminus R_\lambda(\Gamma)$}.$$
  As a consequence, since $u_i$ is continuous on the set
  $\overline{\Sigma_\lambda}\subseteq
  \overline{\Omega}\setminus\Gamma$ (remember that, by as\-sum\-ption
  $\lambda\leq\lambda_0+\epsilon < 0$ and $\Gamma\subseteq\{x_1 = 0\}$), we get
  \begin{align*}
   & \frac{1}{2}\int_{\Sigma_\lambda\setminus K}|\nabla 
   (w_i^{(\lambda)})^+|^2\,\phi_h^2\,\psi_k^2\,\d x \leq
   4\|u\|^2_{L^\infty(\Sigma_{\lambda_0+\epsilon})}
   \,\int_{\Sigma_\lambda\setminus K}
   \psi_k^2\,|\nabla\phi_h|^2
   \,\d x\,+ \\[0.2cm] 
   & \qquad\qquad+4\|u\|^2_{L^\infty(\Sigma_{\lambda_0+\epsilon})}
   \int_{\Sigma_\lambda\setminus K}
   \phi_h^2\,|\nabla\psi_k|^2\,\d x \\[0.2cm]
   & \qquad\qquad + \mathbf{c}_f\sum_{j = 1}^m\int_{\Sigma_\lambda\setminus K}
   (w_i^{(\lambda)})^+ (w_j^{(\lambda)})^+
  \,\phi_h^2\,\psi_k^2\,\d x \\[0.2cm]
  & \qquad \big(\text{by \eqref{eq.propertiespsik}, \eqref{eq.properypsikintegral},
  \eqref{eq.propertiesphih} and \eqref{eq.properyphihintegral}}\big) \\[0.2cm]
  & \qquad \leq 16\|u\|^2_{L^\infty(\Sigma_{\lambda_0+\epsilon})}
  \bigg(\frac{1}{h}+\frac{1}{k}\bigg)+
  \mathbf{c}_f\sum_{j = 1}^m\int_{\Sigma_\lambda\setminus K}
   (w_i^{(\lambda)})^+ (w_j^{(\lambda)})^+\,\d x.
  \end{align*}
  Letting $h,k\to\infty$ (and using Fatou's lemma, see 
  \eqref{eq.propertiespsik} and \eqref{eq.propertiesphih}), we then obtain
  \begin{equation} \label{eq.secondostepnumerato}
   \begin{split}
    &\frac{1}{2}\int_{\Sigma_\lambda\setminus K}|\nabla 
   (w_i^{(\lambda)})^+|^2\,\d x  \leq
   \mathbf{c}_f\sum_{j = 1}^m\int_{\Sigma_\lambda\setminus K}
   (w_i^{(\lambda)})^+ (w_j^{(\lambda)})^+\,\d x \\[0.2cm]
   & \qquad\qquad\qquad
   \leq \mathbf{c}_f\|(w_i^{(\lambda)})^+\|_{L^2(\Sigma_\lambda\setminus K)}\sum_{j = 1}^m
   \|(w_j^{(\lambda)})^+\|_{L^2(\Sigma_\lambda\setminus K)}.
   \end{split}
  \end{equation}
  Now, by exploiting \eqref{eq.wilambdaposoutK}, we can apply 
  \eqref{PerPoincare} (for the Sobolev space $H_0^1(\Sigma_\lambda\setminus K)$)
  in the right-hand side
  of \eqref{eq.secondostepnumerato}: this gives
  \begin{equation} \label{eq.terzostepnumerato}
   \begin{split}
    & \frac{1}{2}\|\nabla (w_i^{(\lambda)})^+\|_{L^2(\Sigma_\lambda\setminus K)}^2 
    = \int_{\Sigma_\lambda\setminus K}|\nabla 
   (w_i^{(\lambda)})^+|^2\,\d x \\[0.2cm]
   & \qquad\qquad\leq \theta^2_n(\Sigma_\lambda\setminus K)\,\mathbf{c}_f\,
   \|\nabla (w_i^{(\lambda)})^+\|_{L^2(\Sigma_\lambda\setminus K)}\sum_{j = 1}^m
   \|\nabla (w_j^{(\lambda)})^+\|_{L^2(\Sigma_\lambda\setminus K)},
   \end{split}
   \end{equation}
	Finally, to complete the demonstration
   of assertion (c)
   we observe that, if 
   \begin{equation} \label{eq.touseinterlinea}
   \text{$\|\nabla(w_i^{(\lambda)})^+\|_{L^2(\Sigma_\lambda\setminus K)}
   = 0$},
   \end{equation}
   then \eqref{eq.toprovebeforechooseK} is trivially satisfied.
   If, instead, \eqref{eq.touseinterlinea} does not hold, 
   by 
   \eqref{eq.terzostepnumerato} one has
   \begin{align*}
    \frac{1}{2}\|\nabla (w_i^{(\lambda)})^+\|_{L^2(\Sigma_\lambda\setminus K)}
    \leq \theta^2_n(\Sigma_\lambda\setminus K)\,\mathbf{c}_f\,
   \sum_{j = 1}^m
   \|\nabla (w_j^{(\lambda)})^+\|_{L^2(\Sigma_\lambda\setminus K)},
   \end{align*}
   and this is precisely 
   the desired \eqref{eq.toprovebeforechooseK} . \medskip

   Now that we have proved \eqref{eq.toprovebeforechooseK}, we are ready to complete
   the proof of the present step. To begin with, let $\delta_0 > 0$ be a fixed
   real number such that
   $$K_\delta := \big\{x\in
   \Sigma_{\lambda_0}\setminus R_{\lambda_0}(\Gamma):\,
   \mathrm{dist}\big(x,\de(\Sigma_{\lambda_0}\setminus R_{\lambda_0}(\Gamma)\big)
   \geq \delta\big\}\neq\varnothing, \quad\forall\,\,\delta\in(0,\delta_0].$$
   Moreover, given any $\delta\in(0,\delta_0]$, let 
   $\epsilon_\delta = \epsilon(K_\delta,\lambda_0) \in (0,|\lambda_0|/2)$ be such that
   assertions (a)-to-(c) hold true for every 
   $\lambda\in[\lambda_0,\lambda_0+\epsilon_\delta]$ (note that $K_\delta$ has Lipschitz boundary).
   
   Since $|R_\lambda(\Gamma)| = 0$ for every $\lambda\in\R$
   (both in the case $n = 2$ and in the case $n\geq 3$, see assumption
   (H.2) and, e.g, \cite[Sec. 4.7]{EG}), it is very easy to recognize that
   \begin{equation} \label{eq.alpostodi}
    \begin{array}{c}
    \text{for every $\eta > 0$ there exists} \\[0.1cm]
    \text{$\delta_\eta\in(0,\delta_0)$ such that}
    \end{array}\qquad
     \theta_n(\Sigma_\lambda\setminus K_\delta) < \eta \qquad
    \begin{array}{c}
     \text{for every $0<\delta<\delta_\eta$ and} \\[0.1cm]
     \text{every $\lambda\in[\lambda_0,\lambda_0+\epsilon_\delta]$}.
    \end{array}
   \end{equation}      
    Starting from \eqref{eq.toprovebeforechooseK} and
   performing an induction argument analogous to that 
   in Step I (in which the information
   $\theta_n(\Sigma_t)\to 0$ as $t\to\mathbf{a}_\Omega$
   is replaced by \eqref{eq.alpostodi}), we
   infer the existence of a small $\sigma\in(0,\delta_0)$ 
   and of a real $C_m = C_m(\mathbf{c}_f) > 0$ such that
   \begin{equation} \label{eq.conseginduzioneStepII}
    \|\nabla (w_{i}^{(\lambda)})^+\|_{L^{2}(\Sigma_\lambda\setminus K_\sigma)} \leq 
     C_m\,\theta^{2}_{n}(\Sigma_{\lambda}\setminus K_\sigma)
     \sum_{j\geq i+1} \|\nabla 
    (w_{j}^{(\lambda)})^+\|_{L^{2}(\Sigma_\lambda\setminus K_\sigma)},
   \end{equation}      
   for every $i = 1,\ldots,m-1$ and every 
   $\lambda\in[\lambda_0,\lambda_0+\epsilon_\sigma]$.
   From this, again by arguing exactly as in Step I, we can use a
   backward induction argument to prove that
   \begin{equation} \label{eq.claimfinaleStepII}
    \|\nabla (w_{j}^{(\lambda)})^{+}\|_{L^{2}(\Sigma_{\lambda}\setminus K_\sigma)} \leq 
  \mathcal{C}_m\,\theta^2_n (\Sigma_\lambda\setminus K_\sigma)
  \,\|\nabla (w_{m}^{(\lambda)})^{+}\|_{L^{2}(\Sigma_{\lambda}\setminus K_\sigma)},
   \end{equation}
   for all $j\in\{1,\ldots,m-1\}$ and 
   every $\lambda\in[\lambda_0,\lambda_0+\epsilon_\sigma]$
   (here, as usual, $\mathcal{C}_m > 0$ is a real constant only depending
   on $\mathbf{c}_f$).
   By combining \eqref{eq.conseginduzioneStepII} with
   \eqref{eq.claimfinaleStepII},
   and by possibly shrinking $\sigma$ if necessary, we obtain
   (see also \eqref{eq.daRichiamareinStepII} in the last part of Step I
   and remember that the vector-valued map $W_\lambda$ is con\-ti\-nuo\-us out 
   of $R_\lambda(\Gamma)$, see \eqref{eq.regulWlambda})
   \begin{equation} \label{eq.ultimabeforeconclude}
   \text{$w_i^{(\lambda)}\leq 0$
    on $\Sigma_\lambda\setminus (K_\sigma\cup R_\lambda(\Gamma))$} 
    \quad \text{$\forall\,\,i\in\{1,\ldots,m\}$
   and $\forall\,\,\lambda\in[\lambda_0,\lambda_0+\epsilon_\sigma]$}.
   \end{equation}
   Gathering together
   \eqref{eq.ultimabeforeconclude} and assertion (b),
   we then conclude that
   $$\text{$w_i^{(\lambda)} \leq 0$
   on $\Sigma_\lambda\setminus R_\lambda(\Gamma)$}\quad\text{for every $i\in\{1,\ldots,m\}$
   and every $\lambda\in[\lambda_0,\lambda_0+\epsilon_\sigma]$}.$$
   From this, a last application of the Strong Maximum Principle gives
   (as $\lambda_0 < 0$)
   $$\text{$u_i - u_i^{(\lambda)} = w_i^{(\lambda)}< 0$
   on $\Sigma_\lambda\setminus R_\lambda(\Gamma)$}\quad\text{$\forall\,\,i\in\{1,\ldots,m\}$
   and $\forall\,\,\lambda\in[\lambda_0,\lambda_0+\epsilon_\sigma]$},$$
   but this is contradiction with the definition of $\lambda_0$.
   Hence, $\lambda_0 = 0$. \medskip
   
   \textsc{Step III:} In this step we prove that all the functions $u_1,\ldots,u_m$
   are symmetric with respect to the hyperplane $\mathit{\Pi} = \{x_1 = 0\}$.
   To this end we first observe that, since we know from Step II that
   $\lambda_0 = \sup\mathcal{I} = 0$ and since $W_\lambda$ is continuous
   out of $R_\lambda(\Gamma)$, one has
   \begin{equation} \label{eq.usymmetricI}
    u_i(x_1,x_2,\ldots,x_n) \leq u_i(-x_1,x_2,\ldots,x_n)
   = u_i^{(0)}(x_1,\ldots,x_n),
   \end{equation}
   for every $i\in\{1,\ldots,m\}$ and every $x\in\Omega_0 = \Omega\cap\{x_1<0\}$.
   By applying this result to the vector-valued function $\hat U:\Omega\to\R^m$ defined by
   $$\hat U(x) := U(-x_1,x_2,\ldots,x_n)$$
   (which has the same regularity of $U$ and is a solution \eqref{eq.mainsystem}), we obtain
   \begin{equation} \label{eq.usymmetricII}
    u_i(-x_1,\ldots,x_n) = 
   \hat u_i(x_1,x_2,\ldots,x_n)\leq \hat u_i(-x_1,x_2,\ldots,x_n)
   = u_i(x_1,\ldots,x_n),
   \end{equation}
   for every $i\in\{1,\ldots,m\}$ and every $x\in\Omega_0$.
   By combining \eqref{eq.usymmetricI} with \eqref{eq.usymmetricII} we get
   $$u_i(-x_1,\ldots,x_n) = u_i(x_1,\ldots,x_n),
   \quad \text{$\forall\,\,i\in\{1,\ldots,m\}$ and $\forall\,\,x\in\Omega\cap\{x_1<0\}$},$$
   and this proves that $u_1,\ldots,u_m$ are symmetric with respect
   to $\mathit{\Pi}$. \medskip
   
   \textsc{Step IV:} In this last step we prove \eqref{eq.monotoneui},
   which clearly implies the monotonicity of the functions
    $u_1,\ldots,u_n$
   in the $x_1$-direction on $\Omega\cap\{x_1<0\}$.
   To this we first observe that, again from the fact that $\lambda_0 = \mathcal{I} = 0$
   (see Step II),
   we have 
   $$\text{$w_i^{(\lambda)} = u_i-u_i^{(\lambda)} < 0$ on
   $\Sigma_\lambda\setminus R_\lambda(\Gamma)$}, \qquad
   \text{$\forall\,\,i\in\{1,\ldots,m\}$ and $\forall\,\,\lambda\in(\mathbf{a}_\Omega,0)$}.$$
   Moreover, $w_i^{(\lambda)}\equiv 0$ on the hyperpalen $\mathit{\Pi}_\lambda
   = \{x_1 = \lambda\}$ and, by \eqref{eq.solvedbywilambdaperHopf},
   $$-\Delta w_i^{(\lambda)} + 
   (\mathbf{c}_f-c_{ii}(\cdot;\lambda)\big)w_i^{(\lambda)}\leq 0,\quad
   \text{on $\Sigma_\lambda\setminus R_\lambda(\Gamma)$}$$
   (where $\mathbf{c}_f$ is as in assumption (H.3) and the $c_{ij}(\cdot;\lambda)$'s
   are defined in \eqref{eq.deficij}).
   Since, by the choice of $\mathbf{c}_f$, we have 
   $\mathbf{c}_f-c_{ii}(\cdot;\lambda)\geq 0$ on $\Sigma_\lambda\setminus R_\lambda(\Gamma)$,
   we are entitled to apply the Hopf lemma for $C^1$-subsolutions
   in \cite{PS}
   (see \eqref{eq.regulWlambda} and note that $\Sigma_\lambda\setminus R_\lambda(\Gamma)$
   certainly satisfies the interior ball condition at any point
   of $\mathit{\Pi}_\lambda\cap\Omega$): this gives
   $$0 < \frac{\de w_i^{(\lambda)}}{\de x_1}(x)
   = 2\,\frac{\de u_i}{\de x_1}(x), \quad\text{$\forall\,\,i\in\{1,\ldots,m\}$
   and every $x\in\mathit{\Pi}_\lambda\cap\Omega$},$$
   which clearly implies the desired \eqref{eq.monotoneui}.
   Hence, the proof of Theorem~\ref{thm.MAIN}
   is complete.
 \end{proof}

\end{document}